\documentclass[11pt]{amsart}
\usepackage{fancyhdr}
\usepackage[english]{babel}
\usepackage{amssymb}
\usepackage{mathrsfs}
\usepackage{enumerate}
\usepackage{color}
\usepackage{ifpdf}
\usepackage{tikz}
\usepackage{tikz-cd}
\usetikzlibrary{decorations.pathmorphing,arrows}
\usepackage[initials]{amsrefs}

\usepackage{color}

\newcommand{\bbN}{{\mathbb N}}

\newcommand{\bbR}{{\mathbb R}}

\newcommand{\bbC}{{\mathbb C}}

\newcommand{\scrD}{\mathscr{D}}

\newcommand{\PSL}{\operatorname{PSL}}

\newcommand{\acts}{\curvearrowright}

\newcommand{\id}{\operatorname{id}}

\newcommand{\interior}{\operatorname{int}}

\newcommand{\pr}{\operatorname{pr}}

\newcommand{\supp}{\operatorname{supp}}

\newcommand{\Isom}{\operatorname{Isom}}

\newcommand{\QM}{\operatorname{QM}}

\newcommand{\Stab}{\operatorname{Stab}}
\newcommand{\half}{\frac{1}{2}}

\newcommand{\Leb}{\lambda_{\bbR}}
\newcommand{\dd }{\,{\rm d}}

\newcommand{\PSnu}{{\nu^{\operatorname{PS}}}}

\newcommand{\Ball}[2]{\operatorname{B}\left({#1},{#2}\right)}

\newcommand{\BMSm}{{m^{\operatorname{BMS}}}}

\newcommand{\BMm}{{\mu^{\operatorname{BM}}}}

\newcommand{\barBMm}{{\bar{\mu}^{\operatorname{BM}}}}

\newcommand{\Mm}{{\mu^{\operatorname{M}}}}

\newcommand{\bG}{{\partial \Gamma}}
\newcommand{\dbG}{{\partial^{2}\Gamma}}

\newcommand{\SAT}{\operatorname{(SAT)}}
\newcommand{\MErg}{\operatorname{(MErg)}}
\newcommand{\WM}{\operatorname{(WM)}}
\newcommand{\Erg}{\operatorname{(Erg)}}
\newcommand{\rSAT}{\operatorname{(rSAT)}}
\newcommand{\rMErg}{\operatorname{(rMErg)}}
\newcommand{\rWM}{\operatorname{(rWM)}}
\newcommand{\rErg}{\operatorname{(rErg)}}

\newcommand{\overto}[1]{{\buildrel{#1}\over\longrightarrow}}

\newcommand{\setdef}[2]{ \left\{ {#1}\ :\ {#2} \right\} }
\newcommand{\wt}[1]{\widetilde{#1}}

\newcommand{\Gprod}[3]{{\langle{#1}\mid{#2}\rangle_{#3}}}
\newcommand{\ang}[2]{d_{\partial\Gamma}({#1},{#2})}

\newcommand{\nactn}[2]{{#1}*{#2}}   
\newcommand{\actn}[2]{{#1}\cdot{#2}}   

\newtheorem{theorem}{Theorem}[section]
\newtheorem{thm}[theorem]{Theorem}
\newtheorem{lemma}[theorem]{Lemma}

\newtheorem{cor}[theorem]{Corollary}
\newtheorem{proposition}[theorem]{Proposition}
\newtheorem{prop}[theorem]{Proposition}

\theoremstyle{definition}

\newtheorem{defn}[theorem]{Definition}

\newtheorem{example}[theorem]{Example}
\newtheorem{examples}[theorem]{Examples}
\newtheorem{remark}[theorem]{Remark}
\newtheorem{setup}[theorem]{Setup}

\newtheorem*{ques*}{Question}

\numberwithin{equation}{section}

\begin{document}

\title[Ergodic properties of hyperbolic groups]
{Some ergodic properties of metrics\\ 
on hyperbolic groups}

\author{Uri Bader}
\address{Weizmann Institute, Rehovot}
\email{bader@weizmann.ac.il}
\author{Alex Furman}
\address{University of Illinois at Chicago}
\email{furman@uic.edu}

\date{\today}

\begin{abstract}
	Let $\Gamma$ be a non-elementary Gromov hyperbolic group,
	and $\partial \Gamma$ denote its Gromov boundary.
	Consider $\Gamma$-invariant proper, quasi-convex metric $d$ on $\Gamma$,
	the associated Patterson--Sullivan measure class $[\PSnu]$ on $\bG$, and its square $[\PSnu\times\PSnu]$ on
	$\dbG$ -- the space of distinct pairs of points on the boundary.
	We study ergodicity properties of the $\Gamma$-actions
	on $(\bG,[\PSnu])$ and on $(\dbG,[\PSnu\times\PSnu])$.
    We also prove some ergodic theorems for $\Gamma$-actions
    guided by the geometry of $(\Gamma,d)$.
\end{abstract}

\maketitle


\section{Introduction and Statement of the main results} 
\label{sec:introduction_and_statement_of_the_main_results}

Closed Riemannian manifolds of strictly negative curvature 
have been extensively studied in geometry and dynamics. 
The geodesic flow on the unit tangent bundle to such a manifold has many invariant probability measures, 
among which there is a unique one of maximal entropy. This measure, constructed independently by Bowen and Margulis,
is closely related to the Patterson--Sullivan measures on the boundary of the universal covering of the manifold,
acted upon by the fundamental group of the manifold.

In this paper we consider a broader class of dynamical systems $\Gamma\acts (\bG,[\PSnu])$, 
where $\Gamma$ is a Gromov hyperbolic group, $\bG$ is its boundary, and $[\PSnu]$ is a Patterson--Sullivan measure class
associated to a geometry on $\Gamma$ taken from a large family that includes 
geometric context as above, word metrics, Green metrics etc. (see Setup~\ref{setup} below). 
We shall prove several ergodicity properties in this broader context, generalizing
some known results and obtaining some results that are 
new even in the context of negatively curved manifolds 
(Theorem~\ref{T:rSAT}, Corollaries~\ref{C:WM} and \ref{C:induced}).

\medskip

\begin{setup}[Coarse-geometric framework]\label{setup}\hfill{}\\
	Let $\Gamma$ be a non-elementary Gromov-hyperbolic group acting faithfully on its boundary $\bG$.
	Consider a proper and cocompact isometric $\Gamma$-action on a quasi-convex Gromov-hyperbolic metric space $(M,d_M)$,
	and the left-invariant pseudo-metric $d$ on $\Gamma$, defined by restricting $d_M$ to a $\Gamma$-orbit of some 
	point $o\in M$:
	\[
		d(g_1,g_2):=d_M(g_1o,g_2o)\qquad(g_1,g_2\in \Gamma).
	\]
	Note that different points $o,o'\in M$ define pseudo-metrics $d,d'$ that differ by a bounded amount: 
	$|d(g_1,g_2)-d'(g_1,g_2)|\le 2d_M(o,o')$.
	
	We denote by $D_\Gamma$ the set of all possible pseudo-metrics $d:\Gamma\times\Gamma\to [0,\infty)$ 
	obtained from such isometric $\Gamma$-actions, and let $\scrD_\Gamma=D_\Gamma/\!\!\sim$ be the space of equivalence classes 
	of such pseudo-metrics,
	where $d\sim d'$ if $d-d'$ is bounded. 
	We denote by $[d]\in\scrD_\Gamma$ the equivalence class of $d\in D_\Gamma$.
	\end{setup}
\begin{examples}\label{E:main}\hfill{}\\ 
	The reader can keep in mind the following main classes of examples:
	\begin{itemize}
		\item[(a)] 
		Let $N$ be a closed connected Riemannian manifold of 
            strictly negative sectional curvature,
		$\Gamma=\pi_1(N)$ its fundamental group, $M=\wt{N}$ its universal covering, 
            and $d_M$ on $M$ the 
		Riemannian metric lifted from the given Riemannian structure on $N$.
		We shall refer to these as \textbf{geometric examples}.
		These well studied objects motivate our general discussion. 
		\item[(b)] 
		Non-elementary Gromov hyperbolic groups $\Gamma$, equipped with a word metric $d_S$ 
		associated with a choice of a finite symmetric generating set $S$ for $\Gamma$.
		\item[(c)] 
		Convex cocompact group actions on proper $\operatorname{CAT}(-1)$ spaces. For example, quasi-Fuchsian 
		embedding of a surface group in $\PSL_2(\bbC)=\Isom_+(\mathbf{H}^3)$.
		\item[(d)] 
		Green metric $d_\mu$ associated with a symmetric, finitely supported, generating probability measure 
		$\mu$ on any Gromov hyperbolic group $\Gamma$ (see \cite{BHM1}).
		\item[(e)]
		Anosov subgroup $\Gamma$ in a higher-rank real Lie group $G$, equipped with the restriction to the orbit
		of a certain Finsler metric on the symmetric space $G/K$ of $G$ (see \cite{DK}). 
	\end{itemize}	
\end{examples} 

\medskip

Let $\Gamma$ and $(M,d_M)$ be as in Setup~\ref{setup}. 
Throughout this paper we shall use the following constructions
that depend only on the equivalence class $[d]\in\scrD_\Gamma$ of $d\in D_\Gamma$.
Denote by $\delta_\Gamma$ the \textbf{growth exponent} for $(\Gamma,d)$, given by
\begin{equation}\label{e:growth}
	\delta_\Gamma:=\lim_{R\to\infty} \frac{1}{R} \log\# \left\{ g\in \Gamma \mid d_M(go,o)<R\right\}.
\end{equation}
The limit exists, does not depend on the choice of the base point $o\in M$, or the representative $d\in[d]$,
and (since we assume $\Gamma$ to be non-elementary) is strictly positive: $\delta_\Gamma>0$ 
(cf. \cite{Coo}).

\begin{example}
	In the geometric Example~\ref{E:main}.(a), $\delta_\Gamma$ coincides with the \emph{topological entropy} 
	of the geodesic flow on the unit tangent bundle 
	$T^1N$; this value is achieved by the Kolmogorov--Sinai entropy of a unique invariant probability measure -- the \emph{Bowen--Margulis measure}.
	In the convex cocompact Example~\ref{E:main}.(c), $\delta_\Gamma$ is usually referred to as the \emph{critical exponent} of $\Gamma$
	(cf. \cite{Yue}).
\end{example}	

Generalizing the Patterson--Sullivan theory from the geometric context, one can define analogous measures in the coarse-geometric setting as  
weak-* limits, where $s\searrow \delta_\Gamma$, of the probability measures $\mu_s$ on the compactification $\overline{M}=M\sqcup \partial M$,
with $\mu_s$ given by
\begin{equation}\label{e:PS-ms}
	\mu_s:=\frac{1}{\sum_{h\in \Gamma}
	e^{-s\cdot d_M(ho,o)}}\cdot \sum_{g\in \Gamma} e^{-s\cdot d_M(go,o)}\cdot {\rm Dirac}_{go}.
\end{equation}
It is known that any weak-* limit, $\mu=\lim_{i\to\infty} \mu_{s_i}$ as $s_i\searrow \delta_\Gamma$,
is supported on the boundary $\partial M=\bG$, has no atoms, and any two such limits $\mu$ and $\mu'$
are mutually equivalent with uniformly bounded Radon-Nikodym derivatives $\dd\mu'/\dd\mu\in L^\infty(\mu)$
(cf. \cite{Coo}). 
For general $d\in \scrD_\Gamma$, one does not expect uniqueness for the limit measures, 
but any weak-* limit will do for our purposes, as we are interested only in the \emph{measure class}
of such measures.
This measure class, denoted $[\PSnu]$, is $\Gamma$-invariant.
As the measure class $[\PSnu]$ is atom-free, its square $[\PSnu\times\PSnu]$ is supported on the space
\[
	\dbG:=\setdef{ (\xi,\eta) \in (\bG)^2 }{ \xi\neq \eta }
\]
of distinct pairs of points on the boundary of $\Gamma$.
This space is locally compact, but is not compact.

In the geometric context of Example~\ref{E:main}.(a), $\dbG$ can be identified with the space $T^1M/\bbR$
of unparametrized (but oriented) geodesic lines in $T^1M$ (here $M=\wt{N}$), 
and its extension by $\bbR$ can be identified 
with the parametrized geodesic lines, and thereby with the unit tangent bundle $T^1M$ itself.
This is often called \textit{Hopf parametrization}. 
In this context the $\Gamma$-action on $\partial^2M=\dbG$ preserves an infinite Radon measure, often called
\emph{Bowen--Margulis--Sullivan current}. 
This action is ergodic; this is directly related to the ergodicity
of the geodesic flow on $T^1N$ equipped with the Bowen--Margulis measure, which is usually proved via Hopf argument.
\medskip

In the general coarse-geometric context of Setup~\ref{setup}, 
we do not have a precise analogue for the geodesic flow on $T^1N$, but can consider the
$\Gamma$-action on the boundary $(\bG,[\PSnu])$ and its double $(\dbG,[\PSnu\times\PSnu])$,
and expect them to be ergodic.
This is established in the following theorem;
along the way constructing an analogue of Bowen--Margulis--Sullivan current,
denoted $\BMSm$ below. 
\begin{thm}[Double Ergodicity]\label{T:double-erg}\hfill{}\\
	The diagonal $\Gamma$-action on $(\dbG,[\PSnu\times\PSnu])$ is ergodic. 
	In particular, the $\Gamma$-action on $(\partial\Gamma,[\PSnu])$ is ergodic.
	The measure class $[\PSnu\times\PSnu]$ contains a unique 
	(up to a certain well defined normalization)
	infinite $\Gamma$-invariant Radon measure $\BMSm$.
\end{thm}
The following result is an ergodic theorem for the 
infinite measure preserving action $\Gamma\acts (\dbG,\BMSm)$.
It is formulated in purely geometric terms, using the {\bf almost Busemann cocycle}
(see \S\ref{sub:topological_almost_geodesic_flow}) defined as:
\[
	\sigma(g,\xi):=\limsup_{x\to\xi}\ (d_M(g^{-1}o,x)-d_M(o,x))\qquad (\xi\in\bG,\ g\in\Gamma).
\]
\begin{thm}[Ergodic Theorem]\label{T:erg}\hfill{}\\
	For any $f\in L^1(\dbG,\BMSm)$ for $\BMSm$-a.e. $(\xi,\eta)\in\dbG$ one has the convergences
\[ \lim_{T\to\infty}\frac{1}{T}\cdot \sum_{\setdef{g\in\Gamma}{\sigma(g,\xi)\in[-T,0]}} f(g\xi,g\eta)=
\int_\dbG f \dd \BMSm, \]
\[ \lim_{T\to\infty}\frac{1}{T}\cdot \sum_{\setdef{g\in\Gamma}{\sigma(g,\eta)\in[0,T]}} f(g\xi,g\eta)=
\int_\dbG f \dd \BMSm \]
and
\[\lim_{T\to\infty} \frac{1}{T}\cdot \sum_{\setdef{g\in\Gamma}{\frac{1}{2}(\sigma(g,\eta)-\sigma(g,\xi))\in[0,T]}} f(g\xi,g\eta)=
\int_\dbG f \dd \BMSm. \]
\end{thm}
Consider the above formulas for a continuous function with compact support $f\in C_c(\dbG)$.
Note that while for sufficiently large $T$ the set 
$\setdef{g\in\Gamma}{\sigma(g,\eta)\in[0,T]}$ is infinite, 
only finitely many elements from
this set contribute non-zero summands $f(g\xi,g\eta)$. 
Moreover, the sums grow linearly in $T$, so changing $\sigma$ in a bounded way has no effect on the limit.
This explicit formula illustrates how the BMS measure $\BMSm$ can be directly derived from the metric $d$ on $\Gamma$,
and that it depends only on the class $[d]\in\scrD_\Gamma$ of equivalent metrics.

\medskip

The double ergodicity stated in Theorem~\ref{T:double-erg} can be deduced 
also from the following finer ergodicity property
(the notions that appear in the following result are defined and discussed in \S\ref{sub:erggen}).
\begin{thm}[Metric Ergodicity] \label{T:rSAT}\hfill{}\\
	Given a probability measure $\nu\in [\PSnu]$ in the Patterson--Sullivan class,
	the projection 
	\[
		\pr_i:(\bG\times\bG,\nu\times\nu)\ \overto{}\ (\bG,\nu),\qquad \pr_i(\xi_1,\xi_2)= \xi_i\qquad (i=1,2)
	\]
	are relatively (SAT), and therefore are relatively metrically ergodic.
	In particular, the diagonal $\Gamma$ action on $(\bG\times\bG,\nu\times\nu)$ is metrically ergodic.
\end{thm}
Metric ergodicity is a powerful property. 
It will be used to show the following freeness result.

\begin{thm}[Freeness of the actions] \label{T:ess-free}\hfill{}\\
	The $\Gamma$-actions on $(\dbG,\BMSm)$ and on $(\bG,[\PSnu])$ are essentially free.
\end{thm}

Another consequence of Theorem~\ref{T:rSAT} is the following.
\begin{cor}[Weak Mixing]\label{C:WM}\hfill{}\\
	The action $\Gamma\acts (\dbG,\BMSm)$ is weakly mixing, in the following sense:
	for any ergodic measure-preserving $\Gamma$-action on a probability space $(\Omega,\omega)$
	the diagonal $\Gamma$-action on $(\dbG\times \Omega,\BMSm\times\omega)$ is ergodic. 
\end{cor}
The following is the corresponding improvement of Theorem~\ref{T:erg}.
\begin{thm}[General Ergodic Theorem]\label{T:generg}\hfill{}\\
For any ergodic measure-preserving $\Gamma$-action on a probability space $(\Omega,\omega)$
and for any $f\in L^1(\dbG\times \Omega,\BMSm\times \omega)$ for $\BMSm\times \omega$-a.e. $(\xi,\eta,w)\in\dbG\times \Omega$ one has the convergences
\[ \lim_{T\to\infty}\frac{1}{T}\cdot \sum_{\setdef{g\in\Gamma}{\sigma(g,\xi)\in[-T,0]}} f(g\xi,g\eta,gw)=
\int_{\dbG\times \Omega} f \dd \BMSm\times \omega, \]
\[ \lim_{T\to\infty}\frac{1}{T}\cdot \sum_{\setdef{g\in\Gamma}{\sigma(g,\eta)\in[0,T]}} f(g\xi,g\eta,gw)=
\int_{\dbG\times \Omega} f \dd \BMSm\times \omega \]
and
\[\lim_{T\to\infty} \frac{1}{T}\cdot \sum_{\setdef{g\in\Gamma}{\frac{1}{2}(\sigma(g,\eta)-\sigma(g,\xi))\in[0,T]}} f(g\xi,g\eta,gw)=
\int_{\dbG\times \Omega} f \dd \BMSm\times \omega. \]
\end{thm}
By taking in Theorem~\ref{T:generg} function 
$f=\psi\cdot\phi$ for $\psi\in L^1(\dbG,\BMSm)$ and
$\phi\in L^1(\Omega,\omega)$, we get 
a scheme for ergodic theorems for probability measure preserving actions of hyperbolic groups.
The following is an illustration of this idea obtained for the case $\psi(\xi,\eta)=e^{-\delta_\Gamma\cdot\Gprod{\xi}{\eta}{o}}$.
\begin{cor}[Averaging Scheme]\label{C:avaragescheme}\hfill{}\\
Fix an ergodic measure-preserving $\Gamma$-action on a standard probability space $(\Omega,\omega)$.
For every $\phi\in L^1(\Omega,\omega)$,
for $\BMSm$-a.e. $(\xi,\eta)\in\dbG$,
for $\omega$-a.e. $w \in \Omega$,
one has the convergences
\[ \lim_{T\to\infty}\frac{1}{T}\cdot\frac{ \sum_{\setdef{g\in\Gamma}{\sigma(g,\eta)\in[0,T]}} 
e^{-\delta_\Gamma\cdot\Gprod{\xi}{\eta}{o}}\phi(gw)}{\sum_{\setdef{g\in\Gamma}{\sigma(g,\eta)\in[0,T]}} 
e^{-\delta_\Gamma\cdot\Gprod{\xi}{\eta}{o}}}=
\int_{\Omega} \phi \dd\omega. \]
\end{cor}

In \S\ref{sub:cocycles} we will introduce a certain measurable
cocycle $\tau$, see Equation~(\ref{e:rho-and-F}).
The ergodic $\Gamma$-action on $(\dbG,\BMSm)$ is
extended to a $\Gamma$-action on $\dbG\times \bbR$ 
preserving the infinite measure $\BMSm\times \Leb$, where $\Leb$ is the Lebesgue
measure on $\bbR$,
using the formula
\[ 
    g\cdot (\xi,\eta,t)=(g\xi,g\eta,t+\tau(g,\xi,\eta), \mbox{ for } 
    g\in \Gamma, ~(\xi,\eta)\in \dbG \mbox{ and } t\in \bbR. 
\]
This extension is no longer ergodic.
On the contrary, we have the following.

\begin{thm}[Fundamental domain] \label{C:ess-free}\hfill{}\\
The action of $\Gamma$ on $\dbG\times \bbR$ admits a measured fundamental domain,
that is a measurable subset $\hat{X}\subset \dbG\times \bbR$ such that for $\BMSm\times \Leb$-a.e
point $(\xi,\eta,t) \in \dbG\times \bbR$ there exists a unique $g\in \Gamma$ such that $(g\xi,g\eta,t+\tau(g,\xi,\eta))\in \hat{X}$.
\end{thm}

The normalization of $\BMSm$, mentioned in Theorem~\ref{T:double-erg}, is chosen to ensure that this fundamental domain has measure one: $\BMSm\times \Leb(\hat{X})=1$.
The measure space $(\dbG\times \bbR,\BMSm\times\Leb)$ with the measure-preserving action of $\Gamma\times\bbR$
defines an $\bbR$-flow, denoted $\phi^\bbR$, on the quotient probability space
\[
	(X,\BMm):=(\dbG\times \bbR,\BMSm\times \Leb)/\Gamma.
\] 
This is a measurable analogue of the geodesic flow on the unit tangent bundle $T^1N$ equipped 
with the Bowen--Margulis measure (hence the notation $\BMm$) in the case of the geodesic flow on 
negatively curved manifolds.

We call the flow $(X,\BMm,\phi^\bbR)$ {\em the measured geodesic flow} associated 
with the pair $(\Gamma,[d])$.
This is a fundamental object of the theory and the last section of this paper will be 
devoted to its study.
It follows from Theorem~\ref{T:double-erg} that the measured geodesic flow is always ergodic.
This fact extends ergodicity results that are well known in the geometric 
Examples~\ref{E:main}.(a).
However, in contrast to the geometric examples, this flow may fail to be mixing, or even weakly mixing, in the general setting.
In fact, in the case of a word metric on a free group the flow has a rotation quotient.
Motivated by this phenomenon, we introduce a new object.
We let $(Z,\Mm,\psi^\bbR)$ be the {\em Mackey range of the Radon-Nikodym cocycle of the $\Gamma$-action on 
$(\bG,\PSnu)$} and observe that it is naturally a factor of the measured geodesic flow $(X,\BMm,\phi^\bbR)$ via a factor map
$(X,\BMm) \to (Z,\Mm)$.

\begin{cor} \label{cor:Mackey}
The factor map $(X,\BMm) \to (Z,\Mm)$ is relatively metrically ergodic.
In particular, it is relatively weakly mixing.
\end{cor}

Conjecturally we have that $(X,\BMm) \to (Z,\Mm)$ is relatively mixing over a rotation flow.

\medskip

By Theorem~\ref{C:ess-free}, the flow $(X,m,\phi^\bbR)$ is endowed with a canonical cohomology class of measurable cocycles
\[
	\gamma:\bbR\times X \overto{} \Gamma.
\]

\begin{cor} \label{cor:relmetric}
Given an isometric action of $\Gamma$ on a separable metric space $S$ and a $\Gamma$-equivariant map $f:X\to S$,
that is a map satisfying for every $t\in \bbR$ and a.e $x\in X$, 
\[ f(\phi^t x)=\gamma_{t,x}f(x), \]
there exists a $\Gamma$-fixed point $s\in S$ such that for a.e $x\in X$, $f(x)=s$.
\end{cor}

Given any ergodic probability measure preserving action $\Gamma\acts (\Omega,\omega)$ one can define  
the \textbf{induced flow} $\phi_\gamma^\bbR$ on $X\times \Omega$ by $\phi_\gamma^t(x,p)=(\phi^t(x),\gamma(t,x).p)$.
This flow can also be constructed as the $\bbR$-flow on the quotient
\[
	(\dbG\times \bbR\times\Omega,\BMSm\times\Leb\times\omega)/\Gamma.
\]
Either by Corollary~\ref{C:WM} or Corollary~\ref{cor:relmetric} we get the following.
\begin{cor}\label{C:induced}\hfill{}\\
	Let $(X,\BMm,\phi^\bbR)$ be the measurable geodesic flow associated 
	to $[d]\in\scrD_\Gamma$, let $\Gamma\acts (\Omega,\omega)$ be an ergodic measure-preserving action on a probability space.
	Then the induced flow $\phi_\gamma^\bbR$ on $(X\times\Omega,\BMm\times \omega)$ is ergodic.
\end{cor}
If $\Gamma$ is a uniform (i.e. cocompact) lattice in a rank-one simple Lie group
$G=\Isom(\mathbf{H}^n_K)$ with $K=\mathbb{R}$, $\mathbb{C}$, $\mathbb{H}$,
or $\mathbb{O}$ with $n=2$, and $M=\mathbf{H}^n_K$ is its associated symmetric space,
then one can deduce the above result using Moore ergodicity theorem.
In this case the flow $\phi^\bbR$ is Bernoulli and the induced flow is 
a K-flow (\cite{Dani}, \cite{FW}).  
However, Corollary~\ref{C:induced} seems to be new even for the geodesic flow on negatively 
curved manifolds that are not locally symmetric.

\begin{remark}
See Diagram~(\ref{e:finaldiagram}) and Remark~\ref{rem:important} in Section~\ref{sec:mackey} for a summary of many of the constructions we introduce in this paper and their relations.
\end{remark}

\medskip

\subsection*{Organization of this paper}
In section \ref{sec:ETpreliminaries} we give some necessary ergodic theoretical preliminaries: 
	a discussion of abstract ergodicity properties in \S\ref{sub:erggen},
a discussion of the space of ergodic components in \S\ref{sub:ergcomp}\ and a discussion of the Mackey range of a cocycle in \S\ref{subsec:mackeyrange}.

In section \ref{sec:Gpreliminaries} we give some necessary geometrical preliminaries:
	our notations and conventions for Gromov hyperbolic geometries in \S\ref{sub:gromov_hyperbolic_spaces},
    a construction of an auxiliary \emph{topological almost geodesic flow} in \S\ref{sub:topological_almost_geodesic_flow}
and the important Lemma~\ref{L:contraction}, which describes in the above setting the well known contraction dynamics of the geodesic flow,
in \S\ref{sub:a_geometric_lemma}.

	Section~\ref{sec:measurable_constructions} describes some measure theoretic constructions, including the Bowen-Margulis-Sullivan $\Gamma$-invariant measure $\BMSm$ on $\dbG$ and the measured geodesic flow $(X,\BMm,\phi^\bbR)$.
	
	Section~\ref{sec:erg-via-Ldiff} contains a proof of Theorem~\ref{T:rSAT} based on a Lebesgue differentiation argument.
Theorem~\ref{T:ess-free},	Corollary~\ref{C:WM} and Theorem~\ref{C:ess-free} 
are then 
deduced from Theorem~\ref{T:rSAT}.


Section~\ref{sec:double_ergodicity} contains the proofs of  Theorem~\ref{T:generg} 
and its Corollary~\ref{C:avaragescheme},
giving along the way an alterantive proof of
Corollary~\ref{C:WM}.
Note that Corollary~\ref{C:WM} implies Theorem~\ref{T:double-erg}
and Theorem~\ref{T:generg} implies Theorem~\ref{T:erg}
as special cases.

Finally, in Section~\ref{sec:mackey} we discuss further properties of the measured geodesic flow $(X,\BMSm,\phi^\bbR)$,
we construct its quotient $(Z,\Mm,\psi^\bbR)$ and prove Corollaries~\ref{cor:Mackey}, \ref{cor:relmetric} and \ref{C:induced}.
We will also present Diagram~(\ref{e:finaldiagram}) and Remark~\ref{rem:important} which summarize many of the constructions we introduce in this paper and their relations.

\subsection*{Some remarks}
	This paper continues \cite{F:cg} and fixes two flaws in the latter. 
	First, in the definition of the coarse-geometric context it is essential to explicitly require
	quasi-convexity (as in framework \ref{setup}), because it is not implied by being quasi-isometric to a word metric.
	
	Secondly, it has been assumed in \cite{F:cg} that the double ergodicity (Theorem~\ref{T:double-erg}) 
	was known in the broad coarse-geometric context; but apparently the ergodicity 
    phenomenon was addressed before in this generality.
	One of the motivations for the present paper was to close this gap, adding along the way some details
	claimed in \cite{F:cg}, such as the existence of the finite measure fundamental domain (Proposition~\ref{P:mgf}).
	
	The present paper was motivated by the beautiful work of Garncarek \cite{Ga}, where 
	the double ergodicity (as in Theorem~\ref{T:double-erg}) is needed
	for the classification of the  irreducible unitary representation $\Gamma\to \mathcal{U}(L^2(\bG,[\PSnu]))$,
	extending \cite{Bader+Muchnik}.

    An unpublished earlier version of this paper was circulated and posted on the arXiv in 2017 and it already has a few citation. We warn the reader that the current version is not identical to this earlier version - we have added a 
    few more details and several new statements applications.
    Let us also emphasize that we do not make an attempt here to cover the latest literature on the subject, but we should at least point out the recent
    papers \cite{Reyes} and \cite{CR, CR2} by Reyes and Cantrell--Reyes.
	
\subsection*{Acknowledgements}
	We would like to thank warmly Lukasz Garncarek and Amos Nevo 
	for interesting conversations regarding the subject matter.
	We would also like to acknowledge grant support.
	U. Bader was supported in part by the ISF-Moked grant 2095/15 and the ERC grant 306706.
	A. Furman was supported in part by the NSF grant DMS 2005493, 2020-2025.
	Both authors were supported by BSF USA-Israel Grant 2018258. 
	

\section{Ergodic theoretical preliminaries} 
\label{sec:ETpreliminaries}

\subsection{Notions of Ergodicity} 
\label{sub:erggen}\hfill{}\\
In this subsection we describe several notions of ergodicity for non-singular group actions,
that are discussed in \cite{BF:sr-note}, \cite{BF:icm}, 
see also \cite{GW}.
For our discussion we fix a locally compact secondly countable group $G$.

\medskip

Given a probability measure $\mu$ on a standard Borel space $(X,\mathcal{X})$ we denote by $[\mu]$ the measure class
of $\mu$, i.e. all probability measures $\mu'$ on $(X,\mathcal{X})$ with $\mu'\sim \mu$.
A measurable action $G\times X\to X$ on a standard probability space $(X,\mu)$
is \textbf{measure-class-preserving} (or \textbf{non-singular}) 
if $[g_*\mu]=[\mu]$, i.e. if $g_*\mu\sim\mu$, for every $g\in G$.
We shall also use the notation $G\acts (X,[\mu])$ and consider $(X,[\mu],G)$ as a \textbf{measure class preserving system}.
Let $G\acts (X,[\mu])$ and $G\acts (Y,[\nu])$ be two measure-class-preserving $G$-actions.
A \textbf{quotient} map is a measurable map $p:X\to Y$ with $p_*\mu\sim\nu$ so that $p\circ g=g\circ p$ 
$\mu$-a.e. on $X$ for every $g\in G$.
A measurable set $A\subset X$ is called $G$-\textbf{invariant} if $\mu(gA\triangle A)=0$ for every $g\in G$.

\medskip

\begin{defn}
A measure-class-preserving action $G\acts (X,[\mu])$ is said to be:
\begin{itemize}
	\item
		\textbf{ergodic} if the only $G$-invariant measurable subsets $A\subset X$ are null or conull:
		$\mu(A)=0$ or $\mu(A^c)=0$.
	\item
		\textbf{weakly mixing} if for any ergodic probability measure preserving (p.m.p.) action $G\acts(\Omega,\omega)$ the diagonal $G$-action 
		on $(X\times \Omega,\mu\times\omega)$ is ergodic. 
	\item
		\textbf{metrically ergodic} if given 
		any separable metric space $(S,d)$ and a continuous homomorphism
		$\pi:G\to\Isom(S,d)$, the only a.e. $G$-equivariant measurable maps $F:X\to S$ 
		are essentially constant ones. 
	\item
		\textbf{strongly almost transitive} $\SAT$ if for any measurable $A\subset X$ with $\mu(A)>0$
		and any $\epsilon>0$, there is $g\in G$ with $\mu(gA)>1-\epsilon$.
\end{itemize}
For a $G$-quotient map $p:(X,[\mu])\to (Y,[\nu])$, we say that the quotient map $p$ is
\begin{itemize}
		\item
		\textbf{relatively ergodic} if the only $G$-invariant measurable sets $A\subset X$
		are, up to null sets, pull-backs of $G$-invariant measurable subsets $B\subset Y$.
	\item
		\textbf{relatively weakly mixing} if for any ergodic probability measure preserving $G$-extension 
$(\Omega,\omega)\to (Y,\nu)$ the 
fiber product map $X\times_Y\Omega\to Y$ 
 is relatively ergodic. 
		\item
		\textbf{relatively metrically ergodic}, if for any measurable family $\{(S_y,d_y)\}_{y\in Y}$
		of (separable) metric spaces, with a measurable family 
		\[
			\left\{\pi_y(g):(S_y,d_y)\overto{} (S_{gy},d_{gy})\right\}\qquad (g\in G,\ y\in Y)
		\]
		of isometries with $\pi_y(gh)=\pi_{hy}(g)\circ \pi_y(h)$,
		the only a.e. defined $G$-equivariant measurable maps $\{ F(x)\in S_{p(x)}\}_{x\in X}$ are pull-backs $F=f\circ p$
		of measurable $G$-equivariant family $\{ f(y)\in S_{y}\}_{y\in Y}$.
		\item
		\textbf{relatively (SAT)} if for any measurable set $A\subset X$ with $\mu(A)>0$ and $\epsilon>0$,
		there is $g\in G$ and a positive $\nu$-measure subset $B\subset Y$ so that
		\[
			\mu(A\cap gA\cap p^{-1}(B))>(1-\epsilon)\cdot\mu(p^{-1}(B)).
		\]
\end{itemize}
\end{defn}

We point out that all these concepts depend only on the relevant measure-classes $[\mu]$, $[\nu]$.
Let us record the following implications (see \cite{BF:icm} for more details).
For measure-class-preserving group actions one has
\begin{equation}
	\SAT \Longrightarrow \MErg \Longrightarrow \WM \Longrightarrow \Erg,
\end{equation}
and for equivariant quotient map
\begin{equation} \label{eq:rel}
	\rSAT \Longrightarrow \rMErg \Longrightarrow \rWM \Longrightarrow \rErg
\end{equation}
where $\SAT$, $\Erg$, $\MErg$, $\WM$ denote strongly almost transitivity, ergodicity, metric ergodicity, weak mixing; 
and $\rSAT$, $\rErg$, $\rMErg$, $\rWM$ the corresponding relative notions.
Let us record a direct proof of the following.
\begin{lemma}\label{P:SAT-erg}\hfill{}\\
	If measure-class preserving actions $G\acts (X,[\mu])$, $G\acts (Y,[\nu])$ are such that the projections 
	\[
		\pr_1:X\times Y\to X, \qquad \pr_2:X\times Y\to Y
	\]
	are relatively (SAT), then the diagonal $G$-action on $(X\times Y,[\mu\times\nu])$ is metrically ergodic. 
	In particular, $(X\times Y, [\mu\times\nu],G)$ is ergodic and weakly mixing.
\end{lemma}
\begin{proof}
	Let $(S,d)$ be a separable metric space, a continuous homomorphism $\pi:G\to\Isom(S,d)$ and $F:X\times Y\to S$ 
	a measurable $G$-equivariant map. 
	Using the assumption that $X\times Y\to X$ is $\rSAT$, we shall prove that $F$ descends to a measurable 
	$G$-equivariant map $F_1:X\to S$, i.e. almost everywhere $F(x,y)=F_1(x)$.
	Applying a similar argument to $X\times Y\to Y$, we will conclude that $F(x,y)$ is a.e. constant $s_0\in S$.
	Since $F$ is equivariant, $s_0$ must be $\pi(G)$-fixed. This will show that $X\times Y$ is $\MErg$.
	
	Let us now show that $F:X\times Y\to S$ descends to $X\to S$. 
	By replacing the metric $d$ by $\min(d,1)$, we may assume that $d\le 1$.
	Define a measurable function $\phi:X\to \bbR_+$ by 
	\[
		\phi(x):=\int_Y\int_Y d(F(x,y_1),F(x,y_2))\dd \nu(y_1)\dd \nu(y_2).
	\] 
	It suffices to show that $\mu$-a.e. $\phi(x)=0$. If not, there is $r>0$ so that the set
	\[
		E:=\setdef{x\in X}{\phi(x)>r}\qquad\textrm{has}\qquad \mu(E)>0.
	\] 
	Take $\epsilon,\rho= r/5$. 
	Since $S$ can be covered by countably many $\rho$-balls,
	it follows that there is $s\in S$ so that
	\[
		A:=\setdef{ (x,y) \in E\times Y }{F(x,y)\in \Ball{s}{\rho}}
		\qquad\textrm{has}\qquad \mu\times \nu(A)>0.
	\]
	For a set $C\subset X\times Y$ and $x\in X$ we denote $C_x:=\setdef{y\in Y}{(x,y)\in C}$.
	It follows from the assumption that $X\times Y\to X$ is $\rSAT$ that there is $g\in G$ 
	and subsets $A'\subset A$ and $E'\subset E$ of positive measure, 
	so that $gA'\subset E\times Y$ and $\nu(gA'_x)>1-\epsilon$ for all $x\in E'\subset E$.
	However, this is a contradiction, because 
	\[
		gA'\subset gA\subset F^{-1}\left(\pi(g)\Ball{s}{\rho}\right)=F^{-1}\left(\Ball{\pi(g)s}{\rho}\right)
	\]
	and so for $x\in E'$ one can estimate
	\[
		\phi(x)\le 2\rho \cdot 1 + 1\cdot 2\epsilon <r.
	\]
	This completes the proof that $G\acts (X\times Y,[\mu\times\nu])$ is $\MErg$.
	
	This property clearly implies ergodicity by considering two point metric space $S=\{0,1\}$ 
	with the trivial $G$-action.
	
	Let us show weak-mixing. Consider an ergodic p.m.p. action $G\acts (\Omega,\omega)$.
	Then the unitary $G$-representation $\pi$ on $L^2_0(\Omega,\omega)=L^2(\Omega,\omega)\ominus \bbC$
	has no non-zero invariant vectors. Take $S$ to be the unit sphere of $L^2_0(\Omega,\omega)$.
	Let $E\subset X\times Y\times \Omega$ be a $G$-invariant subset.
	For $(x,y)\in X\times Y$ denote $E_{x,y}:=\setdef{z\in\Omega}{(x,y,z)\in E}$
	and note that the measurable function $X\times Y\to [0,1]$, defined by  
	$(x,y)\mapsto \omega(E_{x,y})$, is $G$-invariant.
	By ergodicity there is a constant $c$ so that a.e. $\omega(E_{x,y})=c$.
	If $c\ne 0,1$ then $F(x,y)=(1_{E_{x,y}}-c)/\sqrt{c\cdot (1-c)}$ is an equivariant 
	map to $S$, which is impossible. 
	Hence $c=0$ or $c=1$, and so $E$ is null or conull in $X\times Y\times \Omega$.
	This completes the proof of the lemma.
\end{proof}


\subsection{The space of ergodic components} 
\label{sub:ergcomp}\hfill{}\\
In this subsection we continue the discussion of the previous subsection, but this time we consider the case of non-ergodic actions.
We recall the notion of the space of ergodic components of an action.
We fix a locally compact second countable group $H$.
Given a measure class preserving system $(X,[\mu],H)$ we may consider the von-Neumann algebra $L^\infty(X,[\mu])$ on which $H$ acts by weak* continuous automorphisms. The space of invariants $L^\infty(X,[\mu])^H$ is a sub-von-Neumann algebra and it could be identified with the $L^\infty$ space associated with a quotient of $(X,[\mu])$ which we denote $(X/\!/ H,[\bar{\mu}])$ and call {\em the space of ergodic components of the system $(X,[\mu],H)$}.
We endow $(X/\!/ H,[\bar{\mu}])$ with the trivial $H$-action and regard the corresponding system $(X/\!/ H,[\bar{\mu}],H)$.
The injection
\[ L^\infty(X/\!/ H,[\bar{\mu}]) \simeq L^\infty(X,[\mu])^H\hookrightarrow L^\infty(X,[\mu]) \]
gives a natural quotient map $\pi_X:(X,[\mu]) \to (X/\!/ H,[\bar{\mu}])$ which is relatively ergodic.
The association of the space of ergodic components and the above quotient map to a system is functorial in the following sense:
given another system $(Y,[\nu],H)$ and a measure class preserving quotient map $p:X\to Y$, we get a natural measure class preserving quotient map
$\bar{p}:(X/\!/ H,[\bar{\mu}]) \to (Y/\!/ H,[\bar{\nu}])$ such that $\pi_Y\circ p=\bar{p}\circ\pi_X$.
In case $X$ and $Y$ are endowed with a commuting measurable action of a locally compact second countable group $G$,
this action descends to the spaces $(X/\!/ H,[\bar{\mu}])$ and $(Y/\!/ H,[\bar{\nu}])$ and the map $\bar{p}$ becomes a quotient map of the systems
$(X/\!/ H,[\bar{\mu}],G)$ and $(Y/\!/ H,[\bar{\nu}],G)$. 
This follows from Mackey's Point Realization Theorem, \cite{MackeyRealization}.
We note that ergodic properties of the map $p$ descends to $\bar{p}$. In particular, we state the following lemma.

\begin{lemma} \label{lem:ercomp}
We fix the group $G$ and $H$ as above and consider a measure class preserving map $p:X\to Y$ between 
measure class preserving systems $(X,[\mu],G\times H)$ and $(Y,[\nu],G\times H)$.
Assume $p$ is $G\times H$-metrically ergodic.
Then the corresponding quotient map between the spaces of $H$-ergodic components, $\bar{p}:(X/\!/ H,[\bar{\mu}]) \to (Y/\!/ H,[\bar{\nu}])$
is $G$-metrically ergodic.
\end{lemma}
 
\begin{proof}
Given a measurable family $\{(S_{\bar{y}},d_{\bar{y}})\}_{\bar{y}\in Y/\!/ H}$
		of separable metric spaces, with a measurable family 
		\[
			\left\{\pi_{\bar{y}}(g):(S_{\bar{y}},d_{\bar{y}})\overto{} (S_{g{\bar{y}}},d_{g{\bar{y}}})\right\}\qquad (g\in G,\ {\bar{y}}\in Y/\!/ H)
		\]
		of isometries with $\pi_{\bar{y}}(gh)=\pi_{h{\bar{y}}}(g)\circ \pi_{\bar{y}}(h)$,
and an a.e. defined $G$-equivariant measurable maps $\{ F(\bar{x})\in S_{\bar{p}(\bar{x})}\}_{\bar{x}\in Y /\!/ H}$
we obtain a corresponding family 
$\{(\tilde{S}_y,\tilde{d}_y)\}_{y\in Y}$ defined by $(\tilde{S}_y,\tilde{d}_y)=(S_{\bar{y}},d_{\bar{y}})$ where $\bar{y}=\pi_Y(y)$
and an a.e. defined measurable maps $\{ \tilde{F}(x)=F\circ \pi_X(x)\in \tilde{S}_{p(x)}=S_{\bar{p}\circ \pi_X(x)}\}_{x\in X}$, which is obviously $G\times H$-equivariant, in fact $H$-invariant.
By the metric ergodicity of $p$ we deduce that $\tilde{F}=\tilde{f}\circ p$
	for some measurable $G\times H$-equivariant family $\{ \tilde{f}(y)\in \tilde{S}_{y}\}_{y\in Y}$.
As $\tilde{f}$ is clearly $H$-invariant, it descends to a measurable $G$-equivariant family $\{ f(\bar{y})\in S_{\bar{y}}\}_{\bar{y}\in Y/\!/ H}$
and we conclude that indeed, $F=f\circ \bar{p}$.
\end{proof}

\subsection{The Mackey range of a cocycle} 
\label{subsec:mackeyrange}\hfill{}\\
In this subsection we continue the discussion began in the former two subsections. We let $G$ and $H$ be locally compact second countable groups.
We let $(X,[\mu],G)$ and $(Y,[\nu],G)$ be measure class preserving systems and $p:X\to Y$ be a $G$-equivariant measure class preserving quotient map.
We also fix  a measurable cocycle $c:G\times Y \to H$.

We use the cocycle $c$ to define a $G\times H$-action on $Y\times H$, where $H$ acts via the right action on the second coordinate and $G$ acts
using the formula
\[ g(y,h)=(gy,c(g,y)h). \]
We endow $Y\times H$ with the product measure class, taking the Haar class on $H$.
Similarly we define a $G\times H$-action on $X\times H$ using the pull-backed cocycle $p^*c:G\times X\to H$,
$p^c(g,x)=c(g,p(x))$.
The construction is functorial in the sense that the map 
\[ 
    p\times \id_H:X\times H\ \overto{}\ Y\times H 
\]
is $G\times H$-equivariant and measure class preserving.
We recall that the Mackey range of $c$ is the space of $G$-ergodic components of $Y\times H/\!/ G$,
endowed with the quotient measure class and the natural $H$-action.
We denote this space by $M(c)$.
Similarly, we obtain the Mackey range of $p^*c$, $M(p^*c)$.
This construction is functorial in sense that we get a natural $H$-map $\bar{p}:M(p^*c)\to M(c)$ making the following diagram commutative.

\begin{equation}\label{e:mackeyrange}
	\begin{tikzcd}
		 &  X\times H \arrow[dd, "p\times\id_H"] \arrow[dl] \arrow[dr] & \\
 M(\theta^*c) \arrow[dd,"\bar{p}"] & & X \arrow[dd,"p"] \\
& Y\times H \arrow[dl] \arrow[dr] & \\
M(c) & & Y
	\end{tikzcd}
\end{equation}

\begin{prop} \label{prop:mackeyrange}
With notations as in Diagram~(\ref{e:mackeyrange}) conditions: 
\begin{enumerate}
    \item $p$ is $G$-metrically ergodic,
    \item $p\times\id_H$ is $G\times H$-metrically ergodic
\end{enumerate}
are equivalent, and imply that $\bar{p}$ is $H$-metrically ergodic.
\end{prop}

\begin{proof}
The fact that the $G\times H$-metrical ergodicity of $p\times\id_\bbR$ 
implies the $G$-metrical ergodicity of $p$
and the $H$-metrical ergodicity of $\bar{p}$ follows from Lemma~\ref{lem:ercomp} 
(observe that $X$ and $Y$ are the spaces of $H$-ergodic components of $X\times H$ 
and $Y\times H$ correspondingly;
$M(c)$ and $M(p^*c)$ are the spaces of $G$-ergodic components of $X\times H$ 
and $Y\times H$ correspondingly).

We are left to show that the $G$-metrical ergodicity of $p$ implies
the $G\times H$-metrical ergodicity of $p\times\id_H$.
We consider a measurable family $\{(S_{y,h},d_{y,h})\}_{(y,h)\in Y\times H}$
of (separable) metric spaces, with a measurable family 
\[
	\left\{\pi_{y,h}(g,h'):(S_{y,h},d_{y,h})\overto{}
	(S_{gy,c(g,y)hh'^{-1}},d_{gy,c(g,y)hh'^{-1}})\right\} 
\]
with the required compatibility condition
and an a.e. defined $G\times H$-equivariant measurable maps 
$\{ F(x,h)\in S_{p(x),h}\}_{(x,h)\in X\times H}$.
For a.e $y\in Y$ we have that for a.e $h,h'\in H$, $S_{y,h}$ and $S_{y,h'}$ 
are isometric by $\pi_{h'^{-1}h}$.
For a generic $y\in Y$, we define $\bar{S}_y$ to be the space of $H$-equivariant sections 
$(y,h) \mapsto \phi(y,h)\in S_{y,h}$.
For two such sections $\phi_1,\phi_2$, we let $\bar{d}_y(\phi_1,\phi_2)$ to be the 
$H$-essentially constant value $d_{y,h}(\phi_1(h),\phi_2(h))$.
We note that for a.e $x\in X$, $\phi(h)=F(x,h)$ is an $H$-equivariant sections 
$(p(x),h) \mapsto \phi(p(x),h)\in S_{p(x),h}$,
thus defines an a.e. defined $G$-equivariant measurable maps $\{ \bar{F}(x)\in S_{p(x)}\}_{x\in X}$.
By the $G$-metric ergodicity of $p$ we deduce that there exists a measurable 
$G$-equivariant family $\{ \bar{f}(y)\in \bar{S}_{y}\}_{y\in Y}$
such that $\bar{F}=\bar{f}\circ p$.
Setting $f(y,h)$ to be the value at $(y,h)$ of the section $\bar{f}(y)$, we easily seen that
$F=f\circ p$.
\end{proof}


\section{Geometrical preliminaries}
\label{sec:Gpreliminaries}

\subsection{Gromov hyperbolic spaces} 
\label{sub:gromov_hyperbolic_spaces}\hfill{}\\
Let $(M,d_M)$ be a metric space.
The \emph{Gromov product} of $x,y\in M$ relative to $z\in M$ is defined by
\[
	\Gprod{x}{y}{z}:=\half\left(d_M(x,z)+d_M(y,z)-d_M(x,y)\right).
\]
Note that $\Gprod{x}{y}{z}\ge 0$.
A metric space $(M,d)$ is \textbf{Gromov hyperbolic} if there exists a constant $C$ so that
\begin{equation}\label{e:Ghyperbolic}
	\Gprod{x}{z}{o}\ge \min\left(\Gprod{x}{y}{o},\Gprod{y}{z}{o}\right)-C
\end{equation}
for all $x,y,z,o\in M$.
One consequence of this is that for some $D_0$ and any $D>D_0$ for any $x_1,x_2,x_3\in M$ the following set of \textbf{almost-medians}
\begin{equation} \label{eq:qmd}
	\QM_D(x_1,x_2,x_3):=\setdef{m\in M}{\Gprod{x_i}{x_j}{m}<D,\ \forall i\ne j\in\{1,2,3\}}
\end{equation}
is non-empty and has uniformly bounded diameter.

A sequence $(x_i)_{i\in \bbN}$ in $M$ is convergent at infinity if $\Gprod{x_i}{x_j}{o}\to+\infty$ as $i,j\to\infty$ for some
(hence any) $o\in M$.
The Gromov \textbf{boundary} $\partial M$ of $M$ is defined as the set of all equivalence classes $\xi=[(x_i)]$
of sequence $(x_i)$ in $M$ convergent at infinity, where $(x_i)\sim (x'_i)$
if $\Gprod{x_i}{x'_i}{o}\to+\infty$ as $i\to\infty$ for some (hence any) $o\in M$.
The fact that this is an equivalence relation and that the definition is independent of the choice of $o\in M$
follows from the hyperbolicity condition.
Denote
\[
	\overline{M}:=M\sqcup\partial M
\]
and, write $x_i\to \xi\in \overline{M}$ to express $\lim d(x_i,\xi)=0$ if $\xi\in M$,
and $[(x_i)]=\xi$ if $\xi\in\partial M$.
Extend the notion of Gromov product to $x,y\in\overline{M}$ and $o\in M$ by
\begin{equation}\label{e:BinGhyp}
	\Gprod{\xi}{\eta}{o}:=\inf_{x_i\to \xi,\, y_i\to \eta} \liminf_{i\to\infty} \Gprod{x_i}{y_i}{o}.
\end{equation}
Then $\Gprod{\xi}{\eta}{o}<+\infty$ iff $\xi\ne \eta$.
We denote by $\partial^2M$ the space of distinct ordered pairs at the boundary
\[
	\partial^2M:=\setdef{(\xi,\eta)\in (\partial M)^2}{\xi\ne\eta}.
\]
Assuming, as we will, that $M$ is a \textbf{proper} metric space, there is a topology
defined on the Gromov boundary $\partial M$ in which it is compact, and $\overline{M}=M\sqcup \partial M$
is a compactification of $M$ containing the latter as an open dense subset.

The topology on $\partial M$ can also be defined by a metric (or rather a family of H\"older equivalent metrics), 
directly related to the Gromov product, as follows.
There exists $\kappa_0>0$ so that for any $0<\kappa<\kappa_0$
there is a metric $\ang{-}{-}$ satisfying for every $\xi\neq\eta$,
\begin{equation}\label{e:vis-metric}
	\ang{\xi}{\eta}=e^{-\kappa\Gprod{\xi}{\eta}{o}+O(1)}.
\end{equation}
Hereafter $O(1)$ represents an implicit uniformly bounded quantity. 
We fix from now on such a $\kappa\in (0,\kappa_0)$ and the corresponding metric $\ang{-}{-}$ and denote by 
$\Ball{\xi}{r}=\setdef{\eta\in\bG}{\ang{\xi}{\eta}<r}$ the $r$-ball around $\xi$.

\medskip

Define the \textbf{upper Busemann} function $\beta^*:M\times M\times\partial M\to\bbR$ by
\begin{equation}\label{e:beta*}
	\begin{split}
		\beta^*(x,y;\zeta) &:=\limsup_{z\to\zeta}\ \left(d_M(x,z)-d_M(y,z)\right)\\
			&=\limsup_{z\to\zeta}\ \left(d_M(x,y)-2\Gprod{x}{z}{y}\right).
	\end{split}
\end{equation}
Replacing $\ \limsup\ $ by $\ \liminf\ $ we can define a similar 
lower Busemann function $\beta_*:M\times M\times\partial M\to\bbR$.
The hyperbolicity assumption implies that $\beta^*-\beta_*$ is uniformly bounded.
Hereafter we work with the upper Busemann function $\beta^*$ but, as everything will take place up to a bounded error,
this choice makes no essential difference.
Notice that the expression
\[
	\beta^*(x,z;\zeta)=\beta^*(x,y;\zeta)+\beta^*(y,z;\zeta)+O(1)
\]
holds for $x,y,z\in M$, $\zeta\in\partial M$. So $\beta^*(-.-;\zeta)$ is an \textbf{almost-cocycle}.
\begin{remark}
	Note that in many situations, including Examples~\ref{E:main}.(a),(c),(d), the lower and the upper Busemann function coincide 
	and this function, $\beta=\beta_*=\beta^*$, is actually a cocycle, i.e. it satisfies:
	\[
		\beta(x,y;\zeta)+\beta(y,z;\zeta)=\beta(x,z;\zeta)\qquad (x,y,z\in M,\ \zeta\in \partial M).
	\]
	In these situations one can also construct a second Busemann function $B:M\times\partial^2M\to \bbR_+$, denoted $B_x(\xi,\eta)$, 
	with the property 
	\begin{equation}\label{e:B-B=beta+beta}
		\beta(x,y;\xi)+\beta(x,y;\eta)=B_x(\xi,\eta)-B_y(\xi,\eta)\qquad (x,y\in M,\ \xi\ne \eta\in \partial M).
	\end{equation}
	However, the slight simplification of our arguments obtained in these special situation does not justify the restriction of
	generality of the discussion.
	Hence we proceed with the almost-cocycle $\beta^*$ and with the use of Gromov product $\Gprod{-}{-}{x}$ instead of $B_x(-,-)$,
	because one always has
	\begin{equation}\label{e:BBbb}
		\beta^*(x,y;\xi)+\beta^*(x,y;\eta)=\Gprod{\xi}{\eta}{x}-\Gprod{\xi}{\eta}{y}+O(1)
	\end{equation}
	for all $x,y\in M$ and $\xi\ne \eta\in \partial M$.
\end{remark}

\medskip

For a constant $C$ a $C$-\textbf{almost-geodesic} is  a map $p:I\to M$ from an interval $I\subset \bbR$
to $M$ satisfying
\[
	\left|d_M(p(t),p(s))-|t-s|\right|<C\qquad\qquad (t,s\in I).
\]
If $I=[a,b]$ is a finite interval, we say that $p$ is a $C$-\textbf{almost-geodesic segment} connecting $x=p(a)$ to $y=p(b)$ in $M$.
If $I=[a,\infty)$ we say that $p$ is a $C$-almost-geodesic \textbf{ray} connecting $p(a)\in M$ to $\xi=\lim_{t\to\infty} p(t)\in\partial M$;
similarly a $C$-almost-geodesic \textbf{line} $p:(-\infty,+\infty)\to M$ connects the
two points at the boundary $\xi=\lim_{t\to -\infty} p(t)$
and $\eta=\lim_{t\to \infty} p(t)$.
We say that $(M,d_M)$ is \textbf{quasi-convex} if there exists a constant $C$ so that
every two distinct points in $\overline{M}=M\sqcup\partial M$ can be connected by a $C$-almost-geodesic.


\subsection{A topological geodesic almost-flow} 
\label{sub:topological_almost_geodesic_flow}\hfill{}\\
In the setting of Example~\ref{E:main}.(a) we have the geodesic flow action of $\mathbb{R}$ on the unit tangent bundles $T^1N$
and $T^1\wt{N}$.
Our goal is to construct an analogue of this geodesic flow in the more general coarse-geometric 
setting of Setup~\ref{setup}. 
We will indeed construct a measurable version of this action in the next section;
the present section is devoted to a preliminary topological construction. 
Due to the fact that in coarse-geometric framework various metric properties are well behaved 
only up to an additive constant, our topological construction will only be an \textbf{almost-action} (or a \textbf{coarse action}).

\begin{remark}
In \cite{Mineyev} Mineyev constructs a topological version of the geodesic flow resolving various almost-actions.
We decided to avoid using this machinery, as our topological almost-geodesic flow is only an auxiliary tool 
needed for the measurable geodesic flow discussed below.
\end{remark}

\medskip

We now proceed with $\Gamma<\Isom(M,d_M)$ as in Setup~\ref{setup}. We denote the subadditive norm
\begin{equation}\label{e:norm}
	|g|:=d_M(go,o)\qquad (g\in\Gamma).
\end{equation}
Define $\sigma:\Gamma\times\bG\to\bbR$ by
\begin{equation}\label{e:sigma}
	\begin{split}
		\sigma(g,\xi):= & \beta^*(o,g^{-1}o;\xi) =\limsup_{x\to \xi}\ (d_M(o,x)-d_M(g^{-1}o,x)) \\
			=  & \limsup_{x\to \xi}\ \left(2\Gprod{g^{-1}o}{x}{o}-|g|\right)
	 		= 2\Gprod{g^{-1}o}{\xi}{o}-|g|.
	\end{split}
\end{equation}
Notice that $\sigma$ is an \textbf{almost-cocycle}, namely 
\[
	\sigma(gh,\xi)=\sigma(g,h.\xi)+\sigma(h,\xi)+O(1)
\]
for all $g,h\in\Gamma$ and $\xi\in\bG$. 

\medskip

The \emph{rough skeleton} for the almost geodesic-flow is the space $\partial^2 M\times\bbR$.
It is equipped with an action of $\Phi^\bbR$
\[
	\Phi^s(\xi,\eta,t)=(\xi,\eta,t+s), 
\]
the flip involution
\[
	(\xi,\eta,t)\ \mapsto\ (\eta,\xi,-t),
\]
and an \textbf{almost-action}, denoted by a star $*:\Gamma\times \partial^2 M\times\bbR\ \overto{}\ \partial^2 M\times\bbR$,  
defined by
\begin{equation}\label{e:almost-action}
	\nactn{g}{(\xi,\eta,t)}:=\left(g\xi,g\eta,t+\frac{\sigma(g,\eta)-\sigma(g,\xi)}{2}\right).
\end{equation}
Since $\sigma$ is an almost-cocycle, we have 
\[
	\nactn{g}{\left(\nactn{h}{(\xi,\eta,t)}\right)}=\Phi^{O(1)}\left(\nactn{gh}{(\xi,\eta,t)}\right).
\]
\begin{remark}
	This almost-action is actually an action if $\sigma$ is a cocycle.
	This is the case in the geometric example, where the geodesic flow $\Phi^\bbR$ 
	commutes with the flip and the $\Gamma$-action on $T^1\wt{N}$.
\end{remark}
Observe that in the geometric Example~\ref{E:main}.(a) one has a $\Gamma$-equivariant map
\[
	\partial^2 \wt{N}\times \bbR\ \overto{\cong}\ T^1\wt{N}\ \overto{}\ \wt{N},
\]
where $T^1\wt{N}\overto{}\wt{N}$ is the projection to the base point.
The following Proposition describes the properties of an analogous construction in the general coarse-geometric framework of Setup~\ref{setup}.
\begin{prop}\label{P:equi-map}\hfill{}\\
	There exist constants $C,D<\infty$ such that, upon choosing a base point $o\in M$, there is a map
	\[
		\pi:\ \partial^2 M\times \bbR\ \overto{}\  M
	\]
	satisfying:
	\begin{itemize}
		\item[{\rm (a)}] $\pi(\xi,\eta;-):\ \bbR\ \overto{}\ M$ is a $C$-almost-geodesic connecting $\xi$ to $\eta$.
		\item[{\rm (b)}] $\pi(\xi,\eta,0)\in \QM_D(\xi,\eta,o)$ (see Equation~(\ref{eq:qmd}) for the definition).
		\item[{\rm (c)}] $d\left(\pi\left(\nactn{g}{(\xi,\eta,t)}\right),\,g\pi(\xi,\eta,t)\right)<C$ for all $g\in\Gamma$.
	\end{itemize}
\end{prop}
\begin{proof}
	Recall that our Setup~\ref{setup} assumes $M$ to be quasi-convex. 
	Fix $C$ large enough to guarantee an existence of $C$-almost geodesics between any two points of $\overline{M}$.
	Choose $D$ large enough to ensure that for any three points $\xi,\eta,\zeta\in \overline{M}$, 
	$\QM_D(\xi,\eta,\zeta)$ has non-empty intersection with any
	$C$-almost-geodesic connecting any two of the three points.
	
	For each $(\xi,\eta)\in\partial^2M$ choose a $C$-almost-geodesic $\pi(\xi,\eta,-):\bbR\to M$
	connecting $\xi$ to $\eta$, and adjust its parametrization to ensure $\pi(\xi,\eta,0)\in \QM_D(\xi,\eta,o)$.
	Thus properties (a), (b) are satisfied by construction. 
	
	To show (c), consider $g\in\Gamma$, $(\xi,\eta)\in\dbG$, $t\in\bbR$. Since $g$ is an isometry of $(M,d_M)$,
	both $g\pi(\xi,\eta,-)$ and $\pi(g\xi,g\eta,-)$ are $C$-almost geodesics connecting $g\xi$ to $g\eta$.
	Hence for some $\tau_g\in\bbR$
	\[
		d\left(g\pi(\xi,\eta,t),\pi(g\xi,g\eta,t+\tau_g)\right)=O(1).
	\]
	Choose $p\in\QM_D(\xi,\eta,o)$ to be an almost projection of $o$ to the almost geodesic line connecting $\xi$ to $\eta$.
	Then $gp\in\QM_D(g\xi,g\eta,go)=g\QM_D(\xi,\eta,o)$ because $g$ is an isometry of $M$.
	Similarly, choose $q\in \QM_D(g\xi,g\eta,o)$. 
	Note that $|\tau_g|=d(q,gp)+O(1)$ and the sign of $\tau_g$ is determined by the order of $q,gp$ on the $C$-almost geodesic $(\xi,\eta)$.
	Thus we have
	\begin{align*}
		\tau_g+O(1)&= \frac{\beta^*(o,go;g\eta)+\beta^*(go,o;g\xi)}{2}\\
		&= \frac{\beta^*(g^{-1}o,o;\eta)-\beta^*(g^{-1}o,o;\xi)}{2}\\
		&= \frac{\sigma(g,\eta)-\sigma(g,\xi)}{2}.
	\end{align*}
	This proves property (c).
\end{proof}


\subsection{A contraction lemma} 
\label{sub:a_geometric_lemma}
In the following Lemma we record a well-known geometric fact, describing the contraction dynamics on the boundary,
corresponding to the dynamics on the stable foliation in the hyperbolic dynamics of the geodesic flow.

\begin{lemma}\label{L:contraction}\hfill{}\\
	Given a compact subset $K\subset \dbG$ there exists $C=C(K)$ such that:
	if $g\in\Gamma$  and $\xi,\eta,\eta'\in\bG$ satisfy
	\[
		 (\xi,\eta),\ (\xi,\eta'),\ (g\xi,g\eta) \in K,\qquad \sigma(g,\xi)<0
	\]
	then, denoting $t=-\sigma(g,\xi)>0$, we have
	\[
		\sigma(g,\eta), \sigma(g,\eta')\in [t-C,t+C],\qquad
		d_\bG(g\eta,g\eta')<e^{-\kappa t +C}.
	\]
\end{lemma}
\begin{proof}
	Choose 
	\[
		p\in\QM_D(\xi,o,\eta),\quad p'\in \QM_D(\xi,o,\eta'),\quad 
		q\in \QM_D(g\xi,o,g\eta).
	\]
	These points should be thought of as approximate nearest point projections
	of the base point $o\in M$ to the almost-geodesics lines $(\xi,\eta)$, $(\xi,\eta')$, $(g\xi,g\eta)$.
	Let us also choose a ``nearly a projection" $q'$ of $o$ to $(g\xi,g\eta')$,
	namely $q'\in \QM_D(g\xi,o,g\eta')$.
	
	Compactness of $K$ means that $d(o,p)$, $d(o,p')$, $d(o,q)$ are bounded by some $R=R(K)$, and so 
	also   
	\[
		d(p,p')\le d(o,p)+d(o,p')<2R.
	\]
	Since $g$ is an isometry,  $g^{-1}q\in \QM_D(\xi,g^{-1}o,\eta)$;
	so is "nearly a projection" of $g^{-1}o$ to $(\xi,\eta)$.
	As $d(g^{-1}o,g^{-1}q)=d(o,q)<R$, we have 
	\[
		\sigma(g,\xi)=-d(p,q)+O(R),\qquad \sigma(g,\eta)=d(p,q)+O(R).
	\]
	So the almost-geodesic rays $[q,\eta)$ and $[q,\eta')$ have a common initial segment $[q,p]\approx[q,p']$
	of length $t-O(R)$.
	The same applies to almost-geodesic rays $[g^{-1}o,\eta)$ and $[g^{-1}o,\eta')$.
	Thus
	\[
		\Gprod{g\eta}{g\eta'}{o}=\Gprod{\eta}{\eta'}{g^{-1}o}\ge t-O(R)
	\]
	which gives the estimate in the lemma by (\ref{e:vis-metric}).
\end{proof}



\section{Measure-theoretic constructions} 
\label{sec:measurable_constructions}

\subsection{Patterson--Sullivan measures} 
\label{sub:measures}\hfill{}\\
We keep the assumption that $\Gamma<\Isom(M,d)$ acts properly cocompactly on a
quasi-convex Gromov hyperbolic space $(M,d)$ as in the Setup~\ref{setup}.
We also recall choosing $\kappa\in (0,\kappa_0)$ and constructing 
the corresponding metric $\ang{\cdot}{\cdot}$ as in (\ref{e:vis-metric}).

\begin{thm}[\cite{BHM1}*{Theorem~2.3} after \cite{Coo}]\label{T:PS-main}\hfill{}\\
	There exists a probability measure $\nu$ on $\bG$ such that 
    \begin{equation} \label{eq:PS}
		\frac{\dd g^{-1}_*\nu}{\dd\nu}(\xi)=e^{\delta_\Gamma\cdot\sigma(g,\xi)+O(1)}   
    \end{equation}
	Any two such measures are equivalent and have bounded
	Radon--Nikodym derivatives. 
	Furthermore, any such measure $\nu$ is $(\delta_\Gamma/\log\alpha)$-Ahlfors regular, 
	that is for every $\xi\in\bG$ and $r>0$,
	\[ 
		\nu(\Ball{\xi}{r})=r^{\delta_\Gamma/\kappa+O(1)}. 
	\]
\end{thm}
Unpacking the definition of the metric $\ang{\cdot}{\cdot}$  
the last equation is equivalent to 
\begin{equation}\label{e:shadow-estimate}
    \nu\setdef{\eta}{\Gprod{\xi}{\eta}{o}>t}=e^{-\delta_\Gamma t+O(1)}. 
\end{equation}
The original works of Patterson and Sullivan for $\Gamma<\Isom(\mathbf{H}^n)$
was extended to strictly negative curvature (cf. Yue \cite{Yue}), 
and to word metrics on general Gromov-hyperbolic groups
by Coorneart \cite{Coo}.
These methods apply to our more general setting \ref{setup}, see \cite{BHM1}.

Measures satisfying the condition \eqref{eq:PS} of Theorem~\ref{T:PS-main} are called PS-measures, after Patterson and Sullivan,
and we will denote  by $[\PSnu]$ the common measure class.
From now on we fix a PS-measure $\nu$ from this class.

We shall need the fact that the PS-measure $\nu$ is Ahlfors regular (see \cite{BHM1}). 
As a consequence it satisfies the following version of Lebesgue differentiation:
\begin{thm}[Lebesgue Differentiation]\label{T:Lebesgue}\hfill{}\\
	Given $f\in L^1(\bG,\nu)$ for $\nu$-a.e. $\xi\in\bG$:
	\[	
		\begin{split}
		0=&\lim_{r \to0}\ \frac{1}{\nu(\Ball{\xi}{r})}\int_{\Ball{\xi}{r}} 
			|f(\zeta)-f(\xi)|\dd \nu(\zeta)\\
		=&	\lim_{t\to\infty}\ e^{\delta_\Gamma t}\cdot \int_{\setdef{\eta}{\Gprod{\zeta}{\xi}{o}>t}} |f(\eta)-f(\xi)|\dd \nu(\eta).		
		\end{split}
	\]
\end{thm}


\subsection{The Bowen-Margulis-Sullivan measure} 
\label{sub:abstract_geodesic_flow}\hfill{}\\
Let us now establish some further properties. We start from an analogue of Sullivan's result.

\begin{proposition}\label{P:invariantBMS}\hfill{}\\
	There exists a $\Gamma$-invariant Radon measure, denoted $\BMSm$, in the measure class $[\PSnu\times \PSnu]$
	on $\partial^2 M=\partial^2\Gamma$. Moreover, $\BMSm$ has the form
	\[
		\dd\BMSm(\xi,\eta)=e^{F(\xi,\eta)}\dd \nu(\xi)\dd \nu(\eta)
	\]
	where $F$ is a measurable function on $(\dbG,[\PSnu\times\PSnu])$ satisfying 	
	\[
		F(\xi,\eta)=\delta_\Gamma\cdot\Gprod{\xi}{\eta}{o}+O(1).
	\]
\end{proposition}

\begin{proof}
	Consider the Radon measure $m_o$ on $\partial^2 M$ defined by
	\[
		dm_o(\xi,\eta)=e^{\delta_\Gamma\cdot \Gprod{\xi}{\eta}{o}}\,\dd\nu(\xi)\,\dd \nu(\eta).
	\]
	It is $\Gamma$-quasi-invariant, and for $g\in\Gamma$ the log of the Radon-Nikodym derivative satisfies
	\[
		\begin{split}
		\delta_\Gamma^{-1}&\cdot\log\frac{\dd g^{-1}_* m_o}{\dd m_o}(\xi,\eta)\\
		&=\Gprod{\xi}{\eta}{g^{-1}o}-\Gprod{\xi}{\eta}{o}
		   +\delta_\Gamma^{-1}\cdot\log\frac{\dd g^{-1}_*\nu}{\dd\nu}(\xi)
		   +\delta_\Gamma^{-1}\cdot\log\frac{\dd g^{-1}_*\nu}{\dd\nu}(\eta)\\
		&=\Gprod{g\xi}{g\eta}{o}-\Gprod{\xi}{\eta}{o}+\sigma(g,\xi)+\sigma(g,\eta)+O(1)
		\end{split}
	\]
	In view of (\ref{e:BBbb}) the latter is uniformly bounded over $g\in\Gamma$ and $(\xi,\eta)\in\partial^2M$.
	We can now invoke the following general fact.
	
	\begin{lemma}\label{L:bddcoh}\hfil{}\\
		Let $\Gamma\acts Y$ be a measurable action on a Borel space, and $c:\Gamma\times Y\to\bbR$ be a Borel cocycle.
		Assume that $c$ is pointwise bounded in $\Gamma$, i.e. 
		\[
			|c(g,y)|\le C(y)<+\infty
		\]
		Then the cocycle $c$ is a coboundary, namely: $c(g,y)=\phi(gy)-\phi(y)$ for
		some Borel function $\phi:Y\to\bbR$ satisfying $|\phi(y)|\le 2 C(y)$. 
		
		In particular, a uniformly bounded Borel cocycle, $c(g,y)= O(1)$, 
		is a coboundary of a bounded Borel function.
	\end{lemma}
	\begin{proof}\hfill{}\\
		Applying $\ -\sup_h\ $ to the cocycle equation $c(g,y)=c(hg,y)-c(h,gy)$, the function
		\[
			\phi(y)=-\sup\setdef{ c(h,y) }{ h\in \Gamma }
		\]
		gives the a.e. identity $c(g,y)=\phi(gy)-\phi(y)$.
	\end{proof}
	Finally we set
	\begin{equation} \label{eq:F}
		F(\xi,\eta):=\delta_\Gamma\cdot\Gprod{\xi}{\eta}{o}+\phi(\xi,\eta)
	\end{equation}
	where the function $\phi\in L^\infty(\partial^2 M,\nu\times\nu)$ is obtained from Lemma~\ref{L:bddcoh}
	applied to the logarithmic Radon--Nikodym cocycle
	\[
		\log\frac{\dd g^{-1}_* m_o}{\dd m_o}(\xi,\eta)
	\]
	over the $\Gamma$-action on $(\partial^2 M, [\PSnu\times\PSnu])$.
	This completes the proof of Proposition~\ref{P:invariantBMS}.
\end{proof}

\begin{remark} \label{rem:BMSmcanonical}
It will follow from the ergodicity of the $\Gamma$ action on $(\partial^2\Gamma,[\PSnu\times \PSnu])$,
i.e Theorem~\ref{T:double-erg}, that a measure $\BMSm$ satisfying the conditions of Proposition~\ref{P:invariantBMS}
is unique up to a multiplicative constant.
A choice of a normalization will be provided in Proposition~\ref{P:mgf}(e) below.
\end{remark}

At this point we make a choice of a measure $\BMSm$ and a measurable function $F$ 
on $\dbG$ satisfying the conditions of Proposition~\ref{P:invariantBMS}.
In view of Remark~\ref{rem:BMSmcanonical}, these choices will be shown to be canonical later on.
The measure $\BMSm$ is called the \textbf{Bowen-Margulis-Sullivan measure}.

\subsection{Some cocycle identities} 
\label{sub:cocycles}\hfill{}\\
Now let us define a measurable function $\rho:\Gamma\times \partial M\to \bbR$ by
\begin{equation}\label{e:def-rho}
	\rho(g,\xi):=\delta_\Gamma^{-1}\cdot \log\frac{\dd g^{-1}_*\nu}{\dd\nu}(\xi)
\end{equation}
and observe the following properties.
\begin{itemize}
	\item 
	$\rho$ is a measurable cocycle: for $[\PSnu]$-a.e. $\xi\in \bG$
	\[
		\rho(gh,\xi)=\rho(g,h\xi)+\rho(h,\xi)\qquad (g,h\in\Gamma).
	\]
	\item 
	$\rho$ is of bounded distance from $\sigma$:
	\[
		\rho(g,\xi)=\sigma(g,\xi)+O(1)
	\]
	\item 
	The square of $\rho$ is a coboundary of a measurable $F$:
	\[
		\rho(g,\xi)+\rho(g,\eta)=\nabla_g F(\xi,\eta)
	\]
	where $\nabla_g F:=F\circ g-F$.
\end{itemize}
The last identity gives the formula
\begin{equation}\label{e:rho-and-F}
	\begin{split}
		\tau(g,(\xi,\eta))&:=\frac{\rho(g,\eta)-\rho(g,\xi)}{2}\\
		&=\rho(g,\eta)-\frac{1}{2}\nabla_g F(\xi,\eta)\\
		&=-\rho(g,\xi)+\frac{1}{2}\nabla_g F(\xi,\eta).
	\end{split}
\end{equation}
This is a measurable \textbf{cocycle} $\Gamma\times\dbG\to\bbR$, i.e. we have an a.e. identity
\[
	\tau(gh,(\xi,\eta))=\tau(g,(h.\xi,h.\eta))+\tau(h,(\xi,\eta))\qquad (g,h\in\Gamma).
\]

For future record we note the following lemma.

\begin{lemma} \label{lem:tausigma}
For every compact subset $K\subset \dbG$ there exists a constant $C\geq 0$ such that for every $g\in \Gamma$, 
the function 
\[ \dbG \to \bbR, \quad (\xi,\eta) \mapsto |\tau(g,(\xi,\eta))+\sigma(g,\xi)| \]
is uniformly bounded by $C$ on the set $K\cap g^{-1}K$.
\end{lemma}

\begin{proof}
Note that $F(\xi,\eta)=\delta_\Gamma\cdot\Gprod{\xi}{\eta}{o}+O(1)$
is globally bounded from below and it is bounded form above on compact subsets of $\dbG$.
It follows that for every compact subset $K\subset \dbG$ there exists a constant $C'\geq 0$ such that for every $g\in \Gamma$, 
the function 
\[ \dbG \to \bbR, \quad (\xi,\eta) \mapsto |\tau(g,(\xi,\eta))+\rho(g,\xi)| = |\frac{1}{2}\nabla_g F(\xi,\eta)| \]
is uniformly bounded by $C'$ on the set $K\cap g^{-1}K$.
We are done by the fact that	$\rho-\sigma$ is uniformly bounded.
\end{proof}

\subsection{The space $\dbG\times \bbR$} 
\label{sub:dbGR}\hfill{}\\
We use the cocycle $\tau$ to define a measurable $\Gamma$-action on the $\bbR$-extension $\dbG\times \bbR$ 
of the $\Gamma$-action on $(\dbG,\BMSm)$ by the formula
\begin{equation}\label{e:Gamma-action}
	\actn{g}{(\xi,\eta,t)}:=\left(g\xi,g\eta,t+\tau(g,(\xi,\eta)\right).
\end{equation}
Since we used an actual (measurable) cocycle $\tau$, rather than an almost cocycle, we obtain a (measurable) 
$\Gamma$-\textbf{action}, namely we have an identity:
\[
	\actn{gh}{(\xi,\eta,t)}=\actn{g}{\left(\actn{h}{(\xi,\eta,t)}\right)}\qquad (g,h\in\Gamma).
\]
The key properties of this action are summarized in the following Proposition.
We denote by $\Leb$ the Lebesgue measure on $\bbR$.
\begin{prop}\label{P:mgf}\hfill{}\\
	The above measurable $\Gamma$-action on $(\dbG\times\bbR,\BMSm\times\Leb)$ 
	has the following properties:
	\begin{itemize}
		\item[{\rm (a)}] It preserves the infinite measure $\BMSm\times\Leb$.
		\item[{\rm (b)}] It commutes with the $\Phi^\bbR$-action $\Phi^s:(\xi,\eta,t)\mapsto(\xi,\eta,t+s)$.
		\item[{\rm (c)}] It commutes with the flip: $(\xi,\eta,t)\mapsto (\eta,\xi,-t)$.
		\item[{\rm (d)}] It is at essentially bounded distance from the $\Gamma$-almost-action (\ref{e:almost-action}); 
		more precisely there exists a function $\ s\in L^\infty(\dbG,\BMSm)$ so that
		\[
			\actn{g}{(\xi,\eta,t)}=\Phi^{s(\xi,\eta)} \left(\nactn{g}{(\xi,\eta,t)}\right)\qquad (g\in\Gamma).
		\]
		\item[{\rm (e)}] 
		There is a measurable precompact subset $\hat{X}\subset \dbG\times\bbR$ that meets a.e. $\Gamma$-orbit once.
		We choose the scaling for $\BMSm$ so that $\BMSm\times\Leb(\hat{X})=1$.
		\item[{\rm (f)}]
		The size $|\Stab_\Gamma(\xi,\eta,t)|$ of the $\Gamma$-stabilizer of $(\xi,\eta,t)\in\dbG\times\bbR$ is an $\BMSm\times \Leb$-essentially bounded function.
		\item[{\rm (g)}]
		The quotient space $X=(\dbG\times\bbR)/\Gamma$ is endowed with a probability measure $\BMm$ and 
		measure-preserving flow $\phi^\bbR$, satisfying $\phi^t\circ \pr=\pr\circ \Phi^t$  a.e.,
		and a flip $w:X\to X$ so that $w\circ \phi^t=\phi^{-t}\circ w$.
		It admits a measure-space isomorphism 
		\[
			(X,\BMm)\to (\hat{X},\BMSm\times\Leb|_{\hat{X}}),\quad x\mapsto\hat{x}.
		\]
	\end{itemize}
\end{prop}
\begin{proof}
	Statements (a), (b) and (c) follow from the definition of the $\Gamma$-action.
	The $\Gamma$-action clearly preserves the $\BMSm\times\Leb$ measure and satisfies (d)
	because $|\rho-\sigma|$ is uniformly bounded.

	Given a measurable subset $A\subset \dbG\times\bbR$ the 
	cardinality of the intersection $f_A(q)=|A\cap \Gamma.q|$
	is a measurable function $f_A:\dbG\times\bbR\to \{0,1,2,\dots,\infty\}$.
	Fix a point $p\in M$,  and a positive real $R$.
	The set $A_R=\pi^{-1}(\Ball{p}{R})\subset \dbG\times\bbR$ is measurable of finite $\BMSm\times\Leb$-measure.
	It follows from (d), Proposition~\ref{P:equi-map}, and the fact that the $\Gamma$-action on $M$ is proper,
	that for all $R$ the function $f_{A_R}$ is essentially bounded.
	On the other hand, as the $\Gamma$-action on $M$ is cocompact,
	for $R$ large enough, for a.e. $q$ we have $f_{A_R}(q)\ge 1$.
	There exists a measurable choice of a point from the finite intersections $A_R\cap \Gamma.q$.
	Such a choice gives a required subset $\hat{X}\subset A_R$ with a.e. $f_{\hat{X}}=1$.
	This proves (e). Statement (f) follows from (d) and the fact that the actions of $\Gamma$ and $\Phi^\bbR$ commute.
	
	Note that given any two measurable subsets $\hat{X}, \hat{Y}\subset \dbG\times\bbR$ 
	with $f_{\hat{X}}=f_{\hat{Y}}=1$ a.e.,
	the map $\hat{X}\to \hat{Y}$ given by $\hat{x}\mapsto \Gamma.\hat{x}\cap \hat{Y}$,
	is a measurable bijection (mod null sets) that is piecewise translation by elements of $\Gamma$.
	Therefore it is measure-preserving.
	It follows that the measure $\BMSm\times\Leb(\hat{X})$ does not depend on a particular choice of $\hat{X}$ with $f_{\hat{X}}=1$.
	In particular, the normalization of $\BMSm$ by the requirement that $\BMSm\times\Leb(\hat{X})=1$, is well-defined.
\end{proof}

\begin{remark}\label{R:other-cocycles}
	Let $\tau':\Gamma\times (\dbG,\BMSm)\to \bbR$ be a cocycle measurably cohomologous to $\tau$, i.e. assume that
	\[
		\tau'(g,(\xi,\eta))-\tau(g,(\xi,\eta))=\nabla_g H(\xi,\eta)
	\]
	for some measurable $H:\dbG\to\bbR$. 
	Then $\tau'$ can be used to define a $\Gamma$-action on $\dbG\times\bbR$,
	which is measurably isomorphic to (\ref{e:Gamma-action}), 
	via $((\xi,\eta,t)\mapsto (\xi,\eta,t+H(\xi,\eta)))$, and still commutes with $\Phi^\bbR$, but not necessarily with the flip.
	In particular, one could use (as in \cite{F:cg}) the measurable cocycle 
	\[
		\rho_1(g,(\xi,\eta)):=-\rho(g,\xi),\qquad \textrm{or}\qquad \rho_2(g,(\xi,\eta)):=\rho(g,\eta).
	\]
	This leads to the same normalization of $\BMSm$ and the same (i.e. measure-theoretically isomorphic)
	action of $\Gamma\times\bbR$ on $(\dbG\times\bbR,\BMSm\times\Leb)$.
\end{remark}

\begin{remark}
Note that in case $\Gamma$ is torsion free, by Proposition~\ref{P:mgf}(e) and (f) we get that $\hat{X}$ is a fundamental domain for the 
$\Gamma$-action on $\dbG\times \bbR$, that is Theorem~\ref{C:ess-free}.
In general, this Theorem will be derived from Theorem~\ref{T:rSAT} in \S\ref{sec:erg-via-Ldiff}.
\end{remark}

\subsection{The measured geodesic flow $X$ and the associated cocycle $\gamma$} 
\label{sub:measured_geodesic_flow}\hfill{}\\
The system $(X,\BMm,\phi^\bbR)$
defined in Proposition~\ref{P:mgf}(g) is regarded as the {\em measured geodesic flow} associated with $\Gamma,[d]$.
The measure $\BMm$ is denoted the {\em Bowen-Margulis measure}.
We fix hereafter a precompact measurable subset $\hat{X}\subset\dbG\times\bbR$ as in Proposition~\ref{P:mgf}
and let $X\to \hat{X}$, $x\mapsto \hat{x}=(x_-,x_+,t_x)$, be the corresponding measurable cross-section of the quotient map
$\dbG\times\bbR\to X$. 
For $x\in X$ we think of $x_-$ and $x_+$ as the end points of the geodesic associated with $x$ and $t_x$ as the position of $x$ on this geodesic.

The measure geodesic flow comes equipped with a certain cocycle $\gamma:\bbR\times X\to \Gamma$.
Its existence is guaranteed by the following lemma,
which also shows some of its properties.
However, the fact that $\gamma$ is actually a cocycle will be proven only later, as a result of Theorem~\ref{C:ess-free}, see Remark~\ref{R:ess-free} below.
This fact will not be used prior to the proof of Theorem~\ref{C:ess-free}.

\begin{lemma}\label{L:like-cocycle}
	There exists a measurable map $\gamma:\bbR\times X\to \Gamma$ with the property that for a.e. $x\in X$
	\begin{equation}\label{e:def-gamma-t-x}
		\Phi^t(\hat{x})\in \gamma_{t,x}^{-1} \hat{X}\qquad (t\in\bbR).
	\end{equation}
	There exists $C$, so that $d(\pi(x_-,x_+,t),\gamma_{t,x}^{-1}o)\le C$ for $\BMm$-a.e. $x\in X$ for all $t\in\bbR$,
	and therefore 
	\[ 
			|\gamma_{t,x}|=|t|+O(1),\qquad
			\lim_{t\to \infty} \gamma_{t,x}^{-1}o=x_+,\qquad\lim_{t\to -\infty} \gamma_{t,x}^{-1}o=x_-
	\] 
	for $\BMSm$-a.e. $x\in X$.
\end{lemma}
\begin{proof}
	By Proposition~\ref{P:mgf}(e) for a.e. $x\in X$ and every $t\in \bbR$ there exists $\gamma\in \Gamma$
	with $\Phi^t(\hat{x})\in \gamma^{-1}\hat{X}$, and by (f) there are at most finitely many such $\gamma$.
	Thus there is a measurable choice of such $\gamma_{x,t}$. 
	The bound $d(\pi(x_-,x_+,t),\gamma_{t,x}^{-1}o)\le C$ follows from Proposition~\ref{P:mgf}(d)
	and the fact that $\hat{X}$ is precompact in $\dbG\times\bbR$ (equivalently, the fact that $\pi(\hat{X})$ is bounded in $M$).
	This implies the bound $|\gamma_{t,x}|=|t|+O(1)$ and the convergence statements follow from
	Lemma~\ref{L:contraction}.
\end{proof}

\begin{remark}
	In the geometric Example~\ref{E:main}.(a), $X$ represents the unit tangent bundle $T^1N$, 
	$\BMm$ is the Bowen--Margulis measure on $T^1N$, 
	and $\phi^\bbR$ is the geodesic flow. 
In this case, 
	the cocycle $\gamma:\bbR\times T^1N\to \pi_1(N,o)$ can be defined as follows.
	Fix a base point $o\in N$, for each $p\in N$ choose in a measurable way a path $\omega_{o,p}$ in $N$ connecting $o$ to $p$.
	Then for $t\in\bbR$ and $x\in T^1N$ let $\gamma_{t,x}\in \pi_1(N,o)$ be the homotopy class of the path obtained by connecting 
	$o$ to the base point of $x$, followed by the geodesic flow for time $t$, and then using the chosen path to connect back to $o$.
\end{remark}


\section{Stronger Ergodicity via Lebesgue differentiation} 
\label{sec:erg-via-Ldiff}

Our goal in this section is to prove Theorem~\ref{T:rSAT} and 
its corollaries: Theorem~\ref{T:ess-free},	Corollaries~\ref{C:WM} and  Theorem~\ref{C:ess-free}.
Note that Theorem~\ref{T:rSAT} in particular implies also the ergodicity of the action $\Gamma\acts (\dbG,[\PSnu\times\PSnu])$,
that is Theorem~\ref{T:double-erg}.

In the proof of Theorem~\ref{T:rSAT} we will use a version of Lebesgue differentiation combined
with Poincar\'e recurrence for the $\phi^{\bbR}$-flow on the probability space $(X,\BMm)$.


\begin{prop}\label{P:Lebesgue}\hfill{}\\
	Given $f\in L^1(\bG,\nu)$, for $m$-a.e. $x\in X$ one has
	\[
		\lim_{t\to\infty} \int_\bG \left|f(\gamma_{t,x}^{-1}\xi)-f(x_+)\right|\dd \nu(\xi) = 0
	\]
	and therefore 
	\[
		\lim_{t\to\infty} \int_\bG f(\gamma_{t,x}^{-1}\xi)\dd \nu(\xi) = f(x_+).
	\]
	In particular, for any measurable $E\subset \bG$, for $m$-a.e. $x\in X$ one has
	\[
		\lim_{t\to\infty}\nu\left(\gamma_{t,x}E\right)= 1_E(x_+).
	\]
\end{prop}

\begin{proof}
	For $t\ge 0$ let $L_t:\bbR\to\bbR$ denote the piecewise linear function 
	\[
		L_t(s)=\left\{\begin{array}{lll} 	
			-t & \textrm{for} & s\le 0\\
			2s-t & \textrm{for} & 0\le s\le t\\
			t & \textrm{for} & t\le s.
			\end{array}\right.
	\]	
	Then for all $x\in X\subset \dbG\times \bbR$, $t\ge 0$, $\xi\in\bG$ one has 
	\[
		\sigma(\gamma_{t,x},\xi)=L_t(\Gprod{x_+}{\xi}{o})+O(1)
	\]
	according to the location of the "projection" $\QM(\xi,x_+,x_-)$ of $\xi$ to 
	the almost geodesic $\pi(x_-,x_+)$.
	This implies that for $t\ge 0$ for a.e. $x\in X$ and $\xi\in\bG$ 
	\begin{equation}\label{e:RN-L}
		\frac{\dd\gamma_{t,x}^{-1}\nu}{\dd\nu}(\xi)
		=e^{\delta_\Gamma\cdot \sigma(\gamma_{t,x},\xi)+O(1)}
		=e^{\delta_\Gamma\cdot L_t(\Gprod{x_+}{\xi}{o})+O(1)}
	\end{equation}
	(Note that the left hand side is defined only a.e.).

	Fix $f\in L^1(\bG,\nu)$.
	By Theorem~\ref{T:Lebesgue} and the fact that the map $(X,\BMm)\to (\bG,\nu)$, $x\mapsto x_+$, is measure class preserving, 
	we have for a.e $x\in X$, 
	\begin{equation}\label{e:LDTB}
		\lim_{s \to \infty}\ \frac{1}{\nu(\Ball{x_+}{e^{-\kappa s}})}
		\int_{\Ball{x_+}{e^{-\kappa s}}} \left|f(\xi)-f(x_+)\right|\dd \nu(\xi)=0.
	\end{equation}
	We let $X_0\subset X$ be the subset consisting of elements $x\in X$ for which Equations~(\ref{e:RN-L}) and (\ref{e:LDTB}) 
	hold. This is a full measure set.
	We will show that for every $x\in X_0$
	\[ 
		\begin{split}
			\lim_{t\to\infty} &\int_\bG \left|f(\gamma_{t,x}^{-1}\xi)-f(x_+)\right|\dd \nu(\xi)\\
			&=\lim_{t\to\infty}\int_\bG \left|f(\xi)-f(x_+)\right|\frac{\dd\gamma_{t,x}^{-1}\nu}{\dd\nu}(\xi)\dd \nu(\xi)=0.
		\end{split}
	\]
	In view of (\ref{e:RN-L}) it is equivalent to showing
	\begin{equation}\label{e:Lt}
		\lim_{t\to\infty}\int_\bG \left|f(\xi)-f(x_+)\right|\cdot e^{\delta_\Gamma\cdot L_t(\Gprod{x_+}{\xi}{o})}\dd\nu(\xi)=0. 
	\end{equation}
	To shorten notations let 
	\[
		A(s)=\setdef{\xi\in\bG}{\Gprod{x_+}{\xi}{o}\ge s},\qquad I_s=\int_{A(s)}|f(\xi)-f(x_+)|\dd\nu(\xi).
	\]
	We shall establish (\ref{e:Lt}) by showing the asymptotic vanishing of the following integrals:
	\begin{eqnarray}
			\label{e:part1}	 &&\int_{\bG\setminus A(0)} |f(\xi)-f(x_+)|\cdot e^{-\delta_\Gamma\cdot t}\dd\nu(\xi),\\
			\label{e:part2}	 &&\int_{A(0)\setminus A(t)}  |f(\xi)-f(x_+)|\cdot 
				e^{\delta_\Gamma\cdot (2\Gprod{x_+}{\xi}{o}-t)}\dd\nu(\xi),\\
			\label{e:part3}	 &&\int_{A(t)} |f(\xi)-f(x_+)|\cdot e^{\delta_\Gamma\cdot t}\dd\nu(\xi)=I_t\cdot e^{\delta_\Gamma\cdot t},
	\end{eqnarray}
	The fact that (\ref{e:part1}) converges to $0$ is immediate.  
	The asymptotic vanishing of (\ref{e:part3}) is a consequence of (\ref{e:LDTB}) and the fact that for some $c$ and $K$ 
	\[
		A(t)\subset \Ball{x_+}{e^{-\kappa t +c}},
		\qquad
		e^{\delta_\Gamma \cdot t}\le \frac{K}{\nu\left(\Ball{x_+}{e^{-\kappa t +c}}\right)}.
	\]
	To see that (\ref{e:part2}) tends to $0$ as $t\to\infty$, we take $n=\lceil t\rceil$ and estimate 
	\[
		\begin{split}
			(\ref{e:part2})&\le \sum_{k=0}^{n-1} \int_{A(k)\setminus A(k+1)} 
				e^{\delta_\Gamma (2k+2-n)}\cdot |f(\xi)-f(x_+)|\dd\nu(\xi)\\
			&= \sum_{k=0}^{n-1} e^{\delta_\Gamma (2k+2-n)} I_k
			= \sum_{k=0}^{n-1} e^{\delta_\Gamma(2+k-n)}\cdot (e^{\delta_\Gamma k} I_k).
		\end{split}
	\]
	Denote $w_{n,k}=e^{\delta_\Gamma(2+k-n)}$, $a_k=e^{\delta_\Gamma k} I_k$, and observe that 
	the Lebesgue differentiation argument (that was used to estimate (\ref{e:part3})) implies that
	\[
		\lim_{k\to\infty} a_k=0.
	\]
	An elementary calculation shows:
	\[
		\forall k,\ \lim_{n\to\infty} w_{n,k}=0,\qquad B=\sup_{n\ge 1}\left(\sum_{k=0}^{n-1}w_{n,k}\right)<+\infty.
	\]
	Taking large $n_0$ and $n\gg n_0$ we can estimate 
	\[
		\sum_{k=0}^{n-1} w_{n,k}a_k\le \sum_{k=0}^{n_0} w_{n,k}a_k + B\cdot (\max_{k>n_0} a_k)
	\]
	and deduce $\left(\sum_{k=0}^{n-1}w_{n,k}a_k\right)\to 0$ as $n\to\infty$.
	This proves (\ref{e:Lt}), and establishes the first claim in the proposition. 
	The second claim follows from the inequality
	\[
		\left|\int_\bG f(\gamma_{t,x}^{-1}\xi)\dd \nu(\xi) - f(x_+)\right|\le
		\int_\bG \left|f(\gamma_{t,x}^{-1}\xi)-f(x_+)\right|\dd\nu(\xi),
	\]
	and applying it to $f=1_{E}$ shows the third claim.
\end{proof}

\begin{proof}[Proof of Theorem~\ref{T:rSAT}]\hfill{}\\
	Denote $\pr_-,\pr_+:\dbG\to\bG$ the projections to the first and second components, and for a set
	$A\subset\dbG$ and $\xi,\eta\in \bG$ denote the slices
	\[
		A^+_\xi:=\setdef{\eta\in\bG}{(\xi,\eta)\in A },\qquad
		A^-_\eta:=\setdef{\xi\in\bG}{ (\xi,\eta)\in A }.
	\]
	To prove relative (SAT) we shall show that given  $A\subset \dbG$ with $\nu\times\nu(A)>0$ and $\epsilon>0$,
	there is $g\in\Gamma$ and a positive measure subset $B\subset \pr_-(A)\cap g\left(\pr_-(A)\right)$ so that
	for all $\xi\in B$
	\[
		\nu\left(g^{-1}A^+_\xi\right)>1-\epsilon.
	\]
	Observe that it suffices to show this claim for any positive measure subset of $A$,
	or any fixed $\Gamma$-translate $g_0A$ of $A$.
	Since $\Gamma$-translates of $X$ cover $\dbG\times\bbR$, up to a null set,
	upon replacing the given $A$ by a subset of some translate $g_0A$, we may assume that
	$X_A:=\setdef{x\in X}{(x_-,x_+)\in A}$ has $\BMm(X_A)>0$.
	By Proposition~\ref{P:Lebesgue} for a full measure subset of $x\in X_A$ we have
	\[
		\lim_{n\to\infty}\nu(\gamma_{n,x}^{-1}A^+_{x_+})=1.
	\]
	Hence there exists $N$, so that the set
	\[
		X_{A,N}:=\setdef{ x\in X_A }{ \forall n\ge N:\quad \nu\left(\gamma_{n,x}^{-1} A^+_{x_+}\right)>1-\epsilon }
	\]
	has $\BMm(X_{A,N})>0$. 
	We can now apply Poincar\'e recurrence theorem on $(X,\BMm,\phi^\bbR)$ to deduce that there exists $n>N$ for which
	\[
		\BMm(\phi^{-n} X_{A,N}\cap X_{A,N})>0.
	\]
	Denote $Y=\phi^{-n} X_{A,N}\cap X_{A,N}$, and consider its partition according to the $\Gamma$-value of $\gamma_{n,x}$:
	\[
		Y=\bigsqcup_{g\in\Gamma}Y_g,\qquad Y_g=\setdef{ x\in Y }{ \gamma_{n,x}^{-1}=g }.
	\]
	Choose $g\in\Gamma$ so that $\nu\times\nu(Y_g)>0$, set $B=\pr_-(Y_g)$, and observe that:
	\begin{itemize}
		\item $\nu(B)>0$, because $\nu\times\nu(Y_g)>0$.
		\item $B\subset \pr_-(A)\cap g^{-1}\left(\pr_-(A)\right)$, because
		 	$B\subset \pr_-(X_{A,N})\subset \pr_-(A)$ and
		 	\[
				g^{-1}B\subset \pr_-(g^{-1} C_g)\subset \pr_-(X_{A,N})\subset\pr_-(A).
			\]
		\item For $\xi\in B$ there is $x\in Y_g$ so that $\xi=x_-$ and, since $\gamma_{n,x} =g$, we have
		\[
			\nu\left(gA^+_{\xi}\right)=\nu\left(\gamma_{n,x}^{-1} A^+_{x_+}\right)>1-\epsilon
		\]
		as required.
	\end{itemize}
This completes the proof that $\pr_-,\pr_+:\dbG\to\bG$ are relatively SAT.
By Equation~(\ref{eq:rel}) it follows that these maps are also relatively metrically ergodic.
The last statement of the Theorem follows from \cite{BF:icm}*{Remark~2.4(1)}.
\end{proof}

\medskip

\begin{proof}[Proof of Theorem~\ref{C:ess-free}]\hfill{}\\
In view of Proposition~\ref{P:mgf}(e), we only need to show that the point stabilizers for the $\Gamma$-action on $\dbG\times\bbR$ are a.e trivial.
They are a.e finite, by Proposition~\ref{P:mgf}(f).
Denoting by $\operatorname{FSub}_\Gamma$ the countable collection of all finite subgroups of $\Gamma$ and endowing it with the conjugation
action of $\Gamma$,
 we obtain an a.e defined measurable $\Gamma$-equivariant map $\Stab:\dbG\times\bbR \to  \operatorname{FSub}_\Gamma$, taking a point to its stabilizer.
As the $\Gamma$ action on $\dbG\times\bbR$ commutes with $\Phi^\bbR$, this map is $\Phi^\bbR$-invariant, thus descends to 
an a.e defined measurable $\Gamma$-equivariant map $\dbG \to  \operatorname{FSub}_\Gamma$.
Endowing $\operatorname{FSub}_\Gamma$ with the discrete metric and using the metric ergodicity of the $\Gamma$-action on $\dbG$, Theorem~\ref{T:rSAT}, we conclude that the image of the latter map is essentially constant and this constant is $\Gamma$-invariant,
that is a finite normal subgroup $N\lhd \Gamma$.
It follows that $N$ acts essentially trivially on $\dbG\times\bbR$, and therefore it also acts essentially trivially on its quotient $\bG$.
As the measure $[\PSnu]$ is fully supported, we conclude that $N$ acts trivially on $\bG$.
But the $\Gamma$-action on $\bG$ is faithful by assumption, thus $N=\{e\}$ and we conclude that $\Stab$ is essentially constant and its image is $\{e\}$.
This shows that the point stabilizers for the $\Gamma$-action on $\dbG\times\bbR$ are indeed a.e trivial.
\end{proof}

\begin{remark} \label{R:ess-free}
By the fact that $\hat{X}$ is a fundamental domain, Equation~(\ref{e:def-gamma-t-x}) in Lemma~\ref{L:like-cocycle}
uniquely defines $\gamma$, and we get that $\gamma:\bbR\times X\to \Gamma$ satisfies the cocycle identity
\[ \gamma_{t+s,x}=\gamma_{t,\phi^s (x)}\gamma_{s,x} \]
for every $t,s\in \bbR$, for $\BMm$-a.e $x\in X$.
\end{remark}

\begin{proof}[Proof of Theorem~\ref{T:ess-free}]\hfill{}\\
We will first show that the $\Gamma$-action on $\dbG$ is essentially free.
We will assume by contradiction that this is not the case.
By ergodicity we get that the stabilizer of a.e point in $\dbG$ is non-trivial.
By Theorem~\ref{C:ess-free} we have that the $\Gamma$-action on $\dbG\times \bbR$ via the cocycle $\tau$ is essentially free.
It follows that points stabilizers in $\dbG$ are embeddable in $\bbR$.
As abelian subgroups of $\Gamma$ are virtually cyclic, 
it follows that points stabilizers in $\dbG$ are infinite cyclic.
We denote by $\operatorname{CSub}_\Gamma$ the countable collection of all infinite cyclic subgroups of $\Gamma$ and 
obtain an a.e defined measurable $\Gamma$-equivariant map $\Stab:\dbG \to  \operatorname{CSub}_\Gamma$, taking a point to its stabilizer.
Arguing, mutatis mutandis, as in the proof of Theorem~\ref{C:ess-free} we obtain an infinite cyclic normal subgroup $N\lhd \Gamma$ that acts trivially
on $\bG$, contradicting the faithfulness assumption of the $\Gamma$-action.
We conclude that, indeed, the $\Gamma$-action on $\dbG$ is essentially free.

We will now show that the $\Gamma$-action on $\bG$ is essentially free.
Fix $g\in \Gamma$ that fixes a set $A\subset \bG$ of positive $\PSnu$-measure.
Then $g$ fixes also the set $A\times A\subset\dbG$ which is of positive $\BMSm$-measure.
By the fact that the $\Gamma$-action on $\dbG$ is essentially free we conclude that $g=e$.
It follows that, indeed, the $\Gamma$-action on $\bG$ is essentially free as well.
\end{proof}


\begin{proof}[Proof of Corollary~\ref{C:WM}]\hfill{}\\
	Combine Theorem~\ref{T:rSAT} with Lemma~\ref{P:SAT-erg}.
\end{proof}


\section{Double Ergodicity via Hopf argument} 
\label{sec:double_ergodicity}

In this section we prove Theorem~\ref{T:generg} and its Corollary~\ref{C:avaragescheme},
giving along the way an alterantive proof of
Corollary~\ref{C:WM}.
Note that Corollary~\ref{C:WM} implies Theorem~\ref{T:double-erg}
and Theorem~\ref{T:generg} implies Theorem~\ref{T:erg}
as special cases.
A main step in our proof is Proposition~\ref{prop:Iprop} below.
In this section we are not realying on any result obtained in Section~\ref{sec:erg-via-Ldiff},
thus the proofs of Corollary~\ref{C:WM} and Theorem~\ref{T:double-erg} we obtain here are independent of the proofs provided in 
Section~\ref{sec:erg-via-Ldiff}.

Throughout this section we fix
an ergodic measure-preserving $\Gamma$-action on a probability space $(\Omega,\omega)$.
In order to show that the product space $(\dbG\times \Omega,\BMSm\times \omega)$
is $\Gamma$-ergodic, we consider its space of ergodic components $Y$, endowed with 
the quotient measure class $[\mu]$.
We view $(\dbG\times \Omega,\BMSm\times \omega)$ as the space of $\bbR$-ergodic components 
of the space $(\dbG\times \Omega\times\bbR,\BMSm\times\Leb\times \omega)$.
The latter space is endowed with a commuting $\Gamma$ and $\bbR$ actions,
where $\bbR$ acts on the $\bbR$ coordinate and the $\Gamma$ action is an extension the action discussed in Section~\ref{sub:dbGR},
using the cocyle $\tau$.
By Proposition~\ref{P:mgf}(e,f) 
we have that the $\Gamma$-action on $(\dbG\times \Omega\times\bbR,\BMSm\times\Leb\times \omega)$
admits a measurable fundamental domain of finite measure,
thus the corresponding quoient space, which we denote $\bar{X}$,
admits a finite $\bbR$-invariant measure, which we denote $\barBMm$.
We identify the sapce of $\bbR$-ergodic components of $(\bar{X},\barBMm)$ with $Y$ and we 
choose $\mu$ in the given measure class $[\mu]$ to be the measure obtianed as the image of $\barBMm$.
We thus obtained the following commutative diagram.
\begin{equation}\label{e:quotients}
	\begin{tikzcd}
		 &  (\dbG\times\bbR\times \Omega,\BMSm\times\Leb\times \omega) 
		\arrow[dl,"\text{p}"']\arrow[dr,"\text{q}"] & \\
		(\bar{X},\barBMm)  \arrow[dr,"\text{u}"] \arrow[densely dotted, bend left]{ur}{\bar{p}}& & (\dbG\times \Omega,\BMSm\times \omega) \arrow[dl,"\text{v}"']\\
& (Y,\mu) & 
	\end{tikzcd}
\end{equation}
where the measures $\barBMm$ and $\mu$ are finite and $\bbR$-invariant and the map $u$ is measure preserving,
the maps $u$ and $q$ are defined as $\bbR$-ergodic components maps and 
the maps $v$ and $p$ are defined as $\Gamma$-ergodic components maps.
Note that all the spaces above are endowed with a natural flip action
and the maps $p,q,u$ and $v$ are equivariant with respect to these flips.
The dotted arrow $\bar{p}$ is obtained by fixing a fundamental domain for the $\Gamma$-action on 
$(\dbG\times\bbR\times \Omega,\BMSm\times\Leb\times \omega)$.
We chose this fundamenatl domain so that its image in $\dbG\times\bbR$ is precompact.
We introduce an explicit coordinate notation for $\bar{p}$ by writing, for $x\in \bar{X}$,

\begin{equation} \label{eq:barp}
\bar{p}(x)=(x_-,x_+,t_x,w_x)\in \dbG\times\bbR\times \Omega.
\end{equation}

\begin{remark}
In view of Corollary~\ref{C:WM}, $Y$ is a trivial space and in view of 
Theorem~\ref{C:ess-free}
we have that $(\bar{X},\barBMm)\simeq (X\times \Omega,\BMm\times \omega)$
where the latter is endowed with the $\bbR$-action asscoaited with the cocycle $\gamma$
defined in Lemma~\ref{L:like-cocycle}, see Remark~\ref{R:ess-free}.
We will not use these facts a priori, as we aim to have this section idependent Section~\ref{sec:erg-via-Ldiff}.
\end{remark}

The quotient maps in $p,q,u$ and $v$ in Diagram~(\ref{e:quotients}) can be used to define the push forward operators between the spaces of signed measures on the Lebesgue spaces
$\dbG\times\bbR$, $X$, $\dbG$ and $Y$ which are absolutely continuous with respect to the measures $\BMSm\times\Leb$, $\BMm$, $\BMSm$ and $\mu$ correspondingly. 
By the Radon-Nikodym Theorem, these spaces are naturally identified with the corresponding $L^1$ spaces, thus we get the operators 
$P,Q,U$ and $V$ in the following commutative diagram.
\begin{equation}\label{e:L1}
	\begin{tikzcd}
		 & L^1(\dbG\times\bbR\times \Omega,\BMSm\times\Leb\times\omega) \arrow[ddl,"\text{P}"']\arrow[ddr,"\text{Q}"] & \\
		& & \\
		L^1(\bar{X},\barBMm)  \arrow[dr,"\text{U}"]& &  L^1(\dbG\times\Omega,\BMSm\times\omega) \arrow[dl,"\text{V}"'] \arrow[densely dotted, bend right]{uul}[swap]{R_\theta} \\
& L^1(Y,\mu)  &
	\end{tikzcd}
\end{equation}

%
%

In view of the explicit constructions of the measures $\BMSm\times\Leb$, $\barBMm$, $\BMSm$ and $\mu$ we can give explicit description of the operators 
$P,Q$ and $U$ as follows.
The operator $P:L^1(\dbG\times\bbR\times \Omega,\BMSm\times\Leb\times\omega)\to L^1(\bar{X},\barBMm)$ is given by summation over 
the $\Gamma$-orbits
\[
			P(f):=\sum_{g\in\Gamma} f\circ g,
\]
the operator $Q:L^1(\dbG\times\bbR\times \Omega,\BMSm\times\Leb\times\omega)\to L^1(\dbG\times\Omega,\BMSm\times\omega) $ is given by integration over 
the $\bbR$-orbits
\[
			Q(f)(\xi,\eta,w):=\int_\bbR f(\xi,\eta,t,w)\dd t
\]
and the operator $U:L^1(\bar{X},\barBMm) \to L^1(Y,\mu)$ is the operator of integration over fibers, or equivalently the conditional expectation operator.
By Birkhoff's ergodic theorem, for $f\in L^1(\bar{X},\barBMm)$,
for every $a,b$ and $\BMm$-a.e. $x \in \bar{X}$ we have 
\begin{equation} \label{eq:Birk}
U(f)(y) = \lim_{T\to\infty} \frac{1}{T}\int_a^{b+T} f(\phi^{-t} x)\dd t,
\end{equation}
where $y=u(x)$.

We stress that the operator $V:L^1(\dbG\times\Omega,\BMSm\times\omega)\to L^1(Y,\mu)$ is not given a priori by an integration over fibers, as the corresponding map
$(\dbG\times\Omega,\BMSm\times\omega)\to(Y,\mu)$ may not have a measure preserving disintegration.
However, we will give in Proposition~\ref{prop:Iprop} an explicit description of $V$ as well.
But first we will explain the dotted arrow in diagram~\eqref{e:L1}.

The subscript $\theta$ in $R_\theta$ denotes a positive \emph{kernel} $\theta$ on $\bbR$,
i.e. a non-negative measurable function $\theta:\bbR\to[0,\infty)$ with 
\[
	\int_\bbR \theta(t)\dd t=1.
\]
Given such a kernel $\theta$ define the operator $R_\theta$ by 
\[
		R_\theta f(\xi,\eta,t,w)=\theta(t)\cdot f(\xi,\eta,w).	
\]

It follows from the definition that the operators $P,Q,U,V$ and $R_\theta$ are all positive, they have norm one
and they are norm preserving on non-negative functions.
We also have that $Q\circ R_\theta=\operatorname{Id}:L^1(\dbG\times\Omega,\BMSm\times\omega)\to L^1(\dbG\times\Omega,\BMSm\times\omega)$
and $V\circ Q=U\circ P$.
From these relations we further get that 
\begin{equation} \label{eq:defU}
V=U\circ P \circ R_\theta :L^1(\dbG\times\Omega,\BMSm\times\omega) \to L^1(Y,\mu)
\end{equation}

We will use quite often the composed operator $P\circ R_\theta:L^1(\dbG\times\Omega,\BMSm\times\omega)\to L^1(X,\BMm)$,
thus we introduce the notation $\bar{f}_\theta:=P\circ R_\theta(f)$ for $f\in L^1(\dbG\times\Omega,\BMSm\times\omega)$.
Explicitly, using the notation of \eqref{eq:barp}, we have for $x\in \bar{X}$
\begin{equation} \label{eq:fbar}
\bar{f}_\theta(x)=\sum_{g\in \Gamma} R_\theta f(g\bar{p}(x))=\sum_{g\in \Gamma} \theta(t_x+\tau(g,x_-,x_+))\cdot f(gx_-,gx_+,gw_x).
\end{equation}
where $\tau$ is the cocycle given in (\ref{e:rho-and-F}).
Specializing for $\theta=1_{[0,1]}$ we abuse our own notation writing
\begin{equation}\label{e:f01}
	\bar{f}_{[0,1]}:=\bar{f}_{1_{[0,1]}}=P\circ R_{1_{[0,1]}} (f)\ \in L^1(\bar{X},\barBMm).
\end{equation}
In view of (\ref{eq:defU}) and (\ref{eq:Birk}), we get an explicit description of $V:L^1(\dbG\times\Omega,\BMSm\times\omega)\to L^1(Y,\mu)$ given by the formula
\begin{equation} \label{eq:V1}
V(f)(y)= \lim_{T\to\infty} \frac{1}{T}\int_a^{b+T} \bar{f}_{[0,1]}(\phi^{-t} x)\dd t, 
\end{equation}
which holds for a given $f\in L^1(\dbG\times\Omega,\BMSm\times\omega)$ for $\barBMm$-a.e $x\in \bar{X}$ and for $y=u(x)$.
The argument of this formula is an element $x\in X$. 
We wish to decribe the operator $V$ using a parameter in $\dbG\times\Omega$.
This will be done by means of the averaging operators $I_a^b$ and $J_a^b$
which we introduce next.
For a $\BMSm\times \omega$-integrable function $f:\dbG\times\Omega\to [0,\infty)$
and a time interval $[a,b] \subset \bbR$ we define the functions $I_a^b(f), J_a^b(f):\dbG\times\Omega\to [0,\infty]$ by
\begin{align*} \label{e:ab-sum}
	I_a^b(f)(\xi,\eta,w) &:= &\frac{1}{b-a}\cdot \sum_{\setdef{g\in\Gamma}{\tau(g,\xi,\eta)\in[a,b]}} f(g\xi,g\eta,gw) \\
J_a^b(f)(\xi,\eta,w) &:= &\frac{1}{b-a}\cdot \sum_{\setdef{g\in\Gamma}{\sigma(g,\xi)\in[a,b]}} f(g\xi,g\eta,gw).
\end{align*}

\begin{prop} \label{prop:Iprop}
For every $f\in L^1(\dbG\times\Omega,\BMSm\times\omega)$ and $\BMSm\times\omega$-a.e point $(\xi,\eta,w)\in \dbG\times\Omega$, 
for every time interval $[a,b] \subset \bbR$
the values $I_a^b(f)(\xi,\eta,w)$ and $J_a^b (f)(\xi,\eta,w)$ are finite
and we have the convergence
\[ V(f)(y)=\lim_{T\to\infty} I_a^{b+T} (f)(\xi,\eta,w)= \lim_{T\to\infty} J_{a-T}^{b} (f)(\xi,\eta,w), \]
where $y=v(\xi,\eta,w)\in Y$.
\end{prop}

We first note the following easy lemma.

\begin{lemma}
Given two kernels $\theta_1,\theta_2$ on $\bbR$ and given $f\in L^1(\dbG\times\Omega,\BMSm\times\omega)$ we have the relation
\[
\bar{f}_{\theta_2*\theta_1}=\int_\bbR \theta_2(s) \cdot \bar{f}_{\theta_1} \circ \phi^{-s} ds.
\]
\end{lemma}

\begin{proof}
Note that $R_{\theta_1} f\circ \Phi^{-s} (\xi,\eta,t,w)=\theta_1(t-s)\cdot f(\xi,\eta,w)$, thus we get
\[
R_{\theta_2*\theta_1}f = \int_\bbR \theta_2(s) \cdot R_{\theta_1} f \circ \Phi^{-s} ds.
\]
The lemma follows by applying $P$ to both sides of this equation.
\end{proof}

By the above lemma the average of $\bar{f}_{[0,1]}$ by the $\phi^\bbR$-flow over the interval $[a,b]$ satisfies the equation
\begin{equation}\label{e:theta}
	\frac{1}{b-a}\int_a^b \bar{f}_{[0,1]}(\phi^{-t} x)\dd t=\bar{f}_{\theta_a^b}(x)
	=\sum_{g\in\Gamma} \theta_a^b(t_x+\tau(g,x_+,x_-))\cdot f(gx_-,gx_+,gw),
\end{equation}
where $\theta_a^b$ is the convolution
\begin{equation} \label{eq:deftheta}
	\theta_a^b=\frac{1}{b-a}\cdot 1_{[a,b]}*1_{[0,1]}.
\end{equation}
We will use this equation in order to prove some bounds on $I_{a}^{b}(f)$.

\begin{lemma} \label{lem:Ibound}
There exists a constant $c\geq 0$ such that for every $\BMSm\times\omega$-integrable functions $f:\dbG\times\Omega\to [0,\infty)$,
for every time interval $[a,b] \subset \bbR$ with $b\geq a+c+1$ and for every $x\in \bar{X}$,
\[ \frac{1}{b-a}\cdot \int_{a+c}^{b-1-c} \bar{f}_{[0,1]}(\phi^t x)\dd t
\le  I_{a}^{b}(f)(x_-,x_+,w_x) 
\le 
\frac{1}{b-a}\cdot \int_{a-1-c}^{b+c} \bar{f}_{[0,1]}(\phi^t x)\dd t. \]
\end{lemma}

\begin{proof}
Recalling that the image of $\bar{p}(\bar{X})$ in $\dbG\times\bbR$ is precompact, we fix a constant $c\geq 0$ such that for every $x\in \bar{X}$,
$|t_x|\leq c$.

We fix a function $f$, an interval $[a,b]$ and a point $x\in \bar{X}$ as above.
We consider the kernel $\theta_a^b$ defined in (\ref{eq:deftheta})
and note that it is a linear interpolation between the value $1/(b-a)$ on $[a+1,b]$ and $0$ on $(-\infty,a]\cup[b+1,\infty)$, 
thus for every $t\in \bbR$ we get
\[ \frac{1}{b-a}\cdot 1_{[a+1,b]}(t) \leq \theta_a^b(t) \leq \frac{1}{b-a}\cdot 1_{[a,b+1]}.
\]
We conclude that for every $t\in \bbR$
\[ 
\frac{1}{b-a}\cdot 1_{[a+1+c,b-c]}(t) \leq \theta_a^b(t_x+t) \leq \frac{1}{b-a}\cdot 1_{[a-c,b+1+c]}(t).
\]
Fixing $g\in \Gamma$ and substituting $t=\tau(g,x_-,x_+)$ we get
\[
    \begin{split}
        \frac{1}{b-a}\cdot &1_{[a+1+c,b-c]}(\tau(g,x_-,x_+))\cdot f(gx_-,gx_+,gw_x)\\
        &\le \theta_a^b(t_x+\tau(g,x_-,x_+))\cdot f(gx_-,gx_+,gw_x)\\
        &\le \frac{1}{b-a}\cdot 1_{[a-c,b+1+c]}(\tau(g,x_-,x_+)) \cdot f(gx_-,gx_+,gw_x).
    \end{split}
\]
Summing over all $g\in \Gamma$ we get
\[
    \begin{split}
	\frac{b-a-1-2c}{b-a}&\cdot I_{a+1+c}^{b-c}(f)(x_-,x_+,w_x)\\
	&\le\sum_{g\in \Gamma}
	\theta_a^b(t_x+\tau(g,x_-,x_+))\cdot f(gx_-,gx_+,gw_x) \\
        &\le
	\frac{b-a+1+2c}{b-a}\cdot  I_{a-c}^{b+1+c}(f)(x_-,x_+,w_x).
        \end{split}
\]
Thus, by (\ref{e:theta}),
\[
    \begin{split}
	\frac{b-a-1-2c}{b-a}&\cdot I_{a+1+c}^{b-c}(f)(x_-,x_+,w_x)
	\le\frac{1}{b-a}\int_a^b \bar{f}_{[0,1]}(\phi^t x)\dd t \\
        &\le
	\frac{b-a+1+2c}{b-a}\cdot \sum_{g\in \Gamma} I_{a-c}^{b+1+c}(f)(x_-,x_+,w_x).
    \end{split}
\]
Rewriting the last inequality, we easily get the required bounds for $I_a^b(f)(x_-,x_+,w_x)$.
\end{proof}

\begin{proof}[Proof of Proposition~\ref{prop:Iprop}]
We fix a positive function $f\in L^1(\dbG\times\Omega,\BMSm\times\omega)$ and 
argue to show that there exists a full measure subset $A_f\subset \dbG\times\Omega$ such that for every point $(\xi,\eta,w)\in A_f$
and
for every time interval $[a,b] \subset \bbR$,
 the value $I_a^b(f)(\xi,\eta)$ is finite
and we have the convergence
\begin{equation} \label{eq:VfI}
V(f)(y)=\lim_{T\to\infty} I_a^{b+T} (f)(\xi,\eta,w),
\end{equation}
where $y=v(\xi,\eta,w)$.

For every $g\in \Gamma$ we consider the function $f\circ g\in L^1(\dbG\times\Omega,\BMSm\times\omega)$ and observe that $V(f\circ g)=V(f)\in L^1(Y,\mu)$.
We let $\bar{X}_0\subset \bar{X}$ be the full measure set for which the formula (\ref{eq:V1}) holds for the countable collection of functions $f\circ g$, $g\in \Gamma$.
Note that this is an $\bbR$-invariant subset of $\bar{X}$.
We consider the preimage of $\bar{X}_0$ in $\dbG\times \bbR\times\Omega$, $p^{-1}(\bar{X}_0)$,
which is $\Gamma\times\bbR$-invariant subset,
and we let $A_f\subset \dbG\times\Omega$ be its image in $\dbG\times\Omega$, that is $A_f=q(p^{-1}(X_0))$,
see Diagram~(\ref{e:quotients}).
We note that $A_f\subset \dbG\times\Omega$ is indeed a full measure subset.

We fix a point $(\xi,\eta,w)\in A_f$ and an interval $[a,b] \subset \bbR$ and argue to prove (\ref{eq:VfI}).
We consider the point $(\xi,\eta,0,w)\in q^{-1}(A_f)=p^{-1}(\bar{X}_0) \subset \dbG\times \bbR\times\Omega$
and we denote $(x_-,x_+,t_x,w_x)=\bar{p}\circ p(\xi,\eta,0,w)$,
that is the corresponding representative in the fixed fundamntal domain.
We fix $h\in \Gamma$ such that 
\[ (\xi,\eta,0,w)=h(x_-,x_+,t_x,w_x)=(hx_-,hx_+,t_x+\tau(h,x_-,x_+),hw_x). \]
Note that $(x_-,x_+,t_x,w_x)\in p^{-1}(\bar{X}_0)$, as this set is $\Gamma$-invariant.
We let $x=p(x_-,x_+,t_x,w_x)=p(\xi,\eta,0,w) \in \bar{X}_0$ be the corresponding image
and we let $y=u(x)=v(\xi,\eta,w) \in Y$.

Setting $s=\tau(h,x_-,x_+)$ and noting that 
\[ 
    \tau(gh,x_-,x_+)=\tau(g,hx_-,hx_+)+\tau(h,x_-,x_+)=\tau(g,\xi,\eta)+s, 
\]
we get the relation
\[ 
    \begin{split}
        I_a^b(f)(\xi,\eta,w)&=I_a^b(f)(hx_-,hx_+,hw_x)\\
        &=\frac{1}{b-a}\cdot \sum_{\setdef{g\in\Gamma}{\tau(gh,x_-,x_+)\in[a,b]}} f(ghx_-,ghx_+,ghw_x)\\
        &=\frac{1}{b-a}\cdot \sum_{\setdef{g\in\Gamma}{\tau(g,x_-,x_+)\in[a+s,b+s]}} f(ghx_-,ghx_+,ghw_x)\\
        &=I_{a+s}^{b+s}(f\circ h)(x_-,x_+,w_x).
    \end{split}
\]
By Lemma~\ref{lem:Ibound}
there exists a constant $c\geq 0$ such that for the functions 
$f\circ h:\dbG\times\Omega\to [0,\infty)$
and for every $a,b$ and $T$ with $b+T\geq a+c+1$,
\[ 
    \begin{split}
    \frac{1}{b-a}\cdot& \int_{a+s+c}^{b+s-1-c+T} \overline{f\circ h}_{[0,1]}(\phi^t x)\dd t\\
    &\le I_{a+s}^{b+s+T}(f\circ h)(x_-,x_+,w_x)\\
    &\le \frac{1}{b-a}\cdot \int_{a+s-1-c}^{b+s+c+T} \overline{f \circ h}_{[0,1]}(\phi^t x)\dd t.
    \end{split}
\]
Since $x$ is in $\bar{X}_0$, the functions $f\circ h$ satisfies the formula (\ref{eq:V1}),
and therefore the left and right hand sides converge to $V(f\circ h)(y)$ as $T$ tends to $\infty$.
We conclude that 
\[ 
    \begin{split}
        \lim_{T\to\infty} &I_a^{b+T}(f)(\xi,\eta,w) = \lim_{T\to\infty} I_{a+s}^{b+s+T}(f\circ h)(x_-,x_+,w_x)\\
        &=V(f\circ h)(y)=V(f)(y).
    \end{split}
\]
This proves (\ref{eq:VfI}).
The finiteness of $I_a^{b+T}(f)(\xi,\eta,w)$ for large $T$ follows, and by the positivity of $f$ we also deduce the 
finiteness of $I_a^{b}(f)(\xi,\eta,w)$.

\medskip

Next we fix a positive function $f\in L^1(\dbG\times\Omega,\BMSm\times\omega)$ and 
argue to show that there exists a full measure subset $B_f\subset \dbG$ such that for every point $(\xi,\eta,w)\in B_f$
and
for every time interval $[a,b] \subset \bbR$,
 the value $J_a^b(f)(\xi,\eta,w)$ is finite
and we have the convergence
\begin{equation} \label{eq:VfJ}
V(f)(y)=\lim_{T\to\infty} J_{a-T}^{b} (f)(\xi,\eta,w),
\end{equation}
where $y=v(\xi,\eta,w)$.

We fix an exhaustion of $\dbG$ by a nested sequence of compact sets $K_1\subset K_2\subset\dots$.
For each $n\in\bbN$ we consider the function $f_n=1_{K_n}\cdot f$.
We let $B_f=\cap_n  A_{f_n}$ and note that it is a full measure set.
By Lemma~\ref{lem:tausigma}, there exist constants $C_n$ 
so that for $(\xi,\eta)\in K_n$ and $g\in\Gamma$ with $(g\xi,g\eta)\in K_n$ one has
\[
	|\tau(g,\xi,\eta)+\sigma(g,\xi)|\le C_n.
\]
It follows that for every $(\xi,\eta,w)\in \dbG$ 
we have the relation
\[
    \begin{split}
        \frac{a-b-2C_n}{b-a}&\cdot I_{-b+C_n}^{-a-C_n}(f_n)(\xi,\eta,w)\le J_a^b(f_n)(\xi,\eta,w)\\
        &\le \frac{a-b+2C_n}{b-a}\cdot I_{-b-C_n}^{-a+C_n}(f_n)(\xi,\eta,w).
    \end{split}
\]
We conclude that for $(\xi,\eta,w)\in B_f \subset A_{f_n}$,
\[
    \lim_{T\to\infty} J_{a-T}^{b} (f_n)(\xi,\eta,w) = V(f_n)(y),	
\]
where $y=v(\xi,\eta,w)$.
We thus get (\ref{eq:VfJ}) by taking $n\to\infty$ and using monotonicity.
The finiteness of $J_{a-T}^{b}(f)(\xi,\eta,w)$ for large $T$ follows, and by the positivity of $f$ we also deduce the 
finiteness of $J_a^{b}(f)(\xi,\eta,w)$.

This completes the proof for positive functions in $L^1(\dbG\times\Omega,\BMSm\times\omega)$.
The case of arbitrary functions follows by linearity.
\end{proof}

Armed with Proposition~\ref{prop:Iprop}, we now give an alternative proof of Corollary~\ref{C:WM}, 
thus also of its special case, Theorem~\ref{T:double-erg}.

\begin{proof}[Proof of Corollary~\ref{C:WM}]
We will prove the corllary by showing $L^1(Y,\mu)$ is one dimensional, consisting only of constant functions,
thus $(Y,\mu)$ is a singleton.
As $V:L^1(\dbG\times\Omega,\BMSm\times\omega)\to L^1(Y,\mu)$ is a bounded surjection, it is enough to show that $V(f)$ is a constant function for 
every $f$ in a certain dense subspace of $L^1(\dbG\times\Omega,\BMSm\times\omega)$.
We will consider the dense subspace consisting of the image of the injection
\[ C_c(\dbG)\otimes L^1(\Omega,\omega) \hookrightarrow L^1(\dbG\times\Omega,\BMSm\times\omega),
\quad \phi\otimes\psi \mapsto \phi\cdot \psi, \]
where $C_c(\dbG)$ is the space of continuous functions with compact support
on $\dbG$.

We thus fix $\phi\in C_c(\dbG)$ and $\psi\in L^1(\Omega,\omega)$,
denote $f=\phi\cdot\psi\in L^1(\dbG\times\Omega,\BMSm\times\omega)$ and argue to show that $V(f)\in L^1(Y,\mu)$ is constant.
We will do so by shwoing that $\bar{f}=V(f)\circ v$ is constant.
Note that $\bar{f}$, the pull back of $V(f)$ to $\dbG\times\Omega$,
is a $\Gamma$-invariant measurable function, but it need not be in $L^1(\dbG\times\Omega,\BMSm\times\omega)$.
We will show that $\bar{f}$ is independent of the $\dbG$-variable, thus it is
a $\Gamma$-invariant function on $\Omega$, hence constant by the ergodicity assumption
on the $\Gamma$-action on $(\Omega,\omega)$.
Recall that the measure $\BMSm$ on $\dbG$ is equivalent to the product measure $\PSnu\times \PSnu$.
In order to show that $\bar{f}$ is independent of the $\dbG$-variable,
we will show that the following condition is satisfied.
\begin{equation} \label{eq:hopf}
\begin{split}
& \mbox{For $\PSnu\times \omega$-a.e. $(\xi,w)\in\bG\times \Omega$,
there exists a $\PSnu$-full measure subset of $\bG$} \\
& \mbox{such that for every $\eta,\eta'$ in this subset,} \\
& \bar{f}(\xi,\eta,w) = \bar{f}(\xi,\eta',w) \quad \mbox{and} \quad \bar{f}(\eta,\xi,w) = \bar{f}(\eta',\xi,w).
\end{split}
\end{equation}
Recalling that the map $u:(\dbG,\BMSm) \to (Y,\mu)$ is flip-equivariant and conjugating with the flip,
the two equations in \eqref{eq:hopf} imply each other,
thus it is enough to establish the first one.
For this we will show that for every compact subset $K\subset \dbG$
which contains the support of $\phi$ in its interior, $\supp(\phi) \subset \interior(K)$,
the following condition is satisfied.
\begin{equation} \label{eq:hopfK}
\begin{split}
& \mbox{For $\PSnu\times \omega$-a.e. $(\xi,w)\in\bG\times \Omega$,
there exists a $\PSnu$-full measure subset of $\bG$} \\
& \mbox{such that for every $\eta,\eta'$ in this subset,} \\
& (\xi,\eta),(\xi,\eta')\in K \quad \Longrightarrow \quad
\bar{f}(\xi,\eta,w) = \bar{f}(\xi,\eta',w).
\end{split}
\end{equation}
Condition \eqref{eq:hopf} indeed follows from condition \eqref{eq:hopfK}, by taking a sequence of arbitrary large such compact subsets $K$.

We thus fix a compact subset $K\subset \dbG$ satisfying $\supp(\phi) \subset \interior(K)$
and argue to prove condition \eqref{eq:hopfK}.
We consider the indicator function $1_K$ 
and the corresponding function $h=1_K\cdot |\psi| \in L^1(\dbG\times\Omega,\BMSm\times\omega)$.
Denoting $\bar{h}=V(h)\circ v$,
we will show that for every $\epsilon>0$ the following condition is satisfied.
\begin{equation} \label{eq:hopfep}
\begin{split}
& \mbox{For $\PSnu\times \omega$-a.e. $(\xi,w)\in\bG\times \Omega$,
there exists a $\PSnu$-full measure subset of $\bG$} \\
& \mbox{such that for every $\eta,\eta'$ in this subset,} \\
& (\xi,\eta),(\xi,\eta')\in K \quad \Longrightarrow \quad
|\bar{f}(\xi,\eta,w) - \bar{f}(\xi,\eta',w)|<\epsilon\cdot\bar{h}(\xi,\eta,w).
\end{split}
\end{equation}
Condition \eqref{eq:hopfK} indeed follows from condition \eqref{eq:hopfep}, by taking a sequence of arbitrary small $\epsilon$'s.

For the rest of the proof we fix $\epsilon>0$ and argue to prove (\ref{eq:hopfep}).
We apply Proposition~\ref{prop:Iprop} to both functions $f$ and $h$,
obtaining an $\BMSm\times\omega$-full measure subset $A\subset \dbG\times\Omega$
such that for every $(\xi,\eta,w)\in A$ and for every $a,b\in  \bbR$,
we have both
\begin{equation} \label{eq:Jab}
\begin{split}
    &\bar{f}(\xi,\eta,w)=\lim_{T\to\infty} J_{a-T}^{b} (f)(\xi,\eta,w) \\
    &\bar{h}(\xi,\eta,w)=\lim_{T\to\infty} J_{a-T}^{b} (h)(\xi,\eta,w). 
\end{split}
\end{equation}
We will show that (\ref{eq:hopfep}) is satsified for every $\xi,w,\eta,\eta'$ such that $(\xi,\eta,w),(\xi,\eta',w)\in A$,
which is enough upon applying Fubini's Theorem to the map
\[ (\dbG\times\Omega,\PSnu\times\PSnu\times\omega) \to 
(\bG\times\Omega,\PSnu\times\omega) \]
associated with the projection on the first coordinate $\dbG\to\bG$.

We now fix $\xi,w,\eta,\eta'$ such that $(\xi,\eta,w),(\xi,\eta',w)\in A$ and argue to show (\ref{eq:hopfep}).
We asume as we may that $(\xi,\eta),(\xi,\eta') \in K$
and recall that we have $\supp(\phi) \subset \interior(K)$.
We claim that there exists $a<0$ such that for every $g\in \Gamma$,
\begin{equation} \label{eq:siga}
	\sigma(g,\xi)< a \qquad\Longrightarrow\qquad |f(g\xi,g\eta,gw)-f(g\xi,g\eta',gw)|\leq \epsilon\cdot h(g\xi,g\eta,gw).
\end{equation}
For $g\in \Gamma$ with $(g\xi,g\eta)\in K$ 
we use Lemma~\ref{L:contraction} and the uniform continuity of $\phi$
to find $a_1<0$ such that for $\sigma(g,\xi)< a_1$
we get
$|\phi(g\xi,g\eta)-\phi(g\xi,g\eta')|\leq \epsilon$,
thus 
\[ |f(g\xi,g\eta,gw)-f(g\xi,g\eta',gw)|\leq |\phi(g\xi,g\eta)-\phi(g\xi,g\eta')|\cdot |\psi(gw)| \]
\[
\leq \epsilon\cdot |\psi(gw)|=\epsilon\cdot 1_K(g\xi,g\eta)\cdot |\psi(gw)|= \epsilon\cdot h(g\xi,g\eta,gw). \]
Interchanging the roles of $\eta$ and $\eta'$ in Lemma~\ref{L:contraction}
and using $\supp(\phi) \subset \interior(K)$, 
we find $a_2<0$ such that for $\sigma(g,\xi)< a_2$,
if $(g\xi,g\eta')\in\supp(\phi)$ then $(g\xi,g\eta)\in K$.
Thus for $g\in \Gamma$ with $(g\xi,g\eta)\notin K$ and
for $\sigma(g,\xi)< a_2$,
we have both $(g\xi,g\eta),(g\xi,g\eta')\notin\supp(\phi)$,
thus 
\[ 
    \begin{split}
        |f(g\xi,g\eta,gw)-f(g\xi,g\eta',gw)|&= |\phi(g\xi,g\eta)\cdot \psi(gw)-\phi(g\xi,g\eta')\cdot \psi(gw)| =0\\
        &=\epsilon\cdot 1_K(g\xi,g\eta)\cdot |\psi(gw)|= \epsilon\cdot h(g\xi,g\eta,gw).
    \end{split}
\]
Setting $a=\min\{a_1,a_2\}$ we indeed get \eqref{eq:siga} for all $g\in \Gamma$,
proving the claim.
Using \eqref{eq:Jab} we now get
\[
\begin{split}
 |\bar{f}(\xi,\eta,w)&-\bar{f}(\xi,\eta',w)|\\
& =~ | \lim_{T\to\infty} J_{a-T}^{a} (f)(\xi,\eta,w)-\lim_{T\to\infty} J_{a-T}^{a} (f)(\xi,\eta',w)|  \\
& \leq~ \lim_{T\to\infty} \frac{1}{T}\cdot \sum_{\setdef{g\in\Gamma}{\sigma(g,\xi)\in[a-T,a]}} |f(g\xi,g\eta,gw)-f(g\xi,g\eta',gw)|  \\
& \leq~ \lim_{T\to\infty} \frac{1}{T}\cdot \sum_{\setdef{g\in\Gamma}{\sigma(g,\xi)\in[a-T,a]}}  \epsilon\cdot h(g\xi,g\eta,gw)  \\
& =~ \lim_{T\to\infty} J_{a-T}^{a} h(\xi,\eta,w) = \epsilon\cdot \bar{h}(\xi,\eta,w).
\end{split}
\]
which proves (\ref{eq:hopfep})
and thus finishes our proof.
\end{proof}

\begin{proof}[Proof of Theorem~\ref{T:generg}]

Recall that $V:L^1(\dbG\times\Omega,\BMSm\times\omega) \to L^1(Y,\mu)$ was defined at the beginning of this section by means of push-forward of absolutely continuous measures.
It follows from  Corollary~\ref{C:WM} that $Y$ is trivial, thus for every $f\in L^1(\dbG\times\Omega,\BMSm\times\omega)$, $V(f)$ is the constant 
$\int_{\dbG\times\Omega} f \dd \BMSm\times\omega$. 
It follows from Proposition~\ref{prop:Iprop} that for $\BMSm\times\omega$-a.e point $(\xi,\eta,w)\in \dbG\times\Omega$, 
\[ \lim_{T\to\infty}\frac{1}{T}\cdot \sum_{\setdef{g\in\Gamma}{\sigma(g,\xi)\in[-T,0]}} f(g\xi,g\eta,gw) \]
\[ =
    \lim_{T\to\infty} J_{-T}^0(f)(\xi,\eta,w)=\int_{\dbG\times\Omega} f \dd \BMSm\times\omega. 
\]
Conjugating with the flip, we obtain
\[ 
    \lim_{T\to\infty}\frac{1}{T}\cdot \sum_{\setdef{g\in\Gamma}{\sigma(g,\eta)\in[0,T]}} f(g\xi,g\eta,gw)=
\int_{\dbG\times\Omega} f \dd \BMSm\times\omega. 
\]
It also follows from Proposition~\ref{prop:Iprop} that for every $a,b\in \bbR$,
\[ 
    \begin{split}
        \lim_{T\to\infty}&\frac{1}{T}\cdot \sum_{\setdef{g\in\Gamma}{\tau(g,\xi,\eta)\in[a,b+T]}} f(g\xi,g\eta,gw)\\
        &=\lim_{T\to\infty} I_a^{b+T}(f)(\xi,\eta,w) \\
        &=\int_{\dbG\times\Omega} f \dd \BMSm\times\omega.
    \end{split}
\]
Note that by the definition of $\tau$, (\ref{e:rho-and-F}), and the discussion preceding it, 
there exists $C\geq 0$ such that for every $g\in \Gamma$ and every $(\xi,\eta)\in \dbG$,
\[ |\tau(g,\xi,\eta)-\frac{1}{2}(\sigma(g,\eta)-\sigma(g,\xi))| \leq C. \]
It follows that for all $T\geq 2C$ and for positive function $f$,
\[ 
    \begin{split}
        \frac{T-2C}{T}&\cdot I_C^{T-C}f(\xi,\eta,w)\\
        &\le
\frac{1}{T}\cdot \sum_{\setdef{g\in\Gamma}{\frac{1}{2}(\sigma(g,\eta)-\sigma(g,\xi))\in[0,T]}} f(g\xi,g\eta,gw) \\
    &\le \frac{T+2C}{T}\cdot I_{-C}^{T+C}f(\xi,\eta,w).
    \end{split}
\]
Taking the limit we conclude that 
\[\lim_{T\to\infty} \frac{1}{T}\cdot \sum_{\setdef{g\in\Gamma}{\frac{1}{2}(\sigma(g,\eta)-\sigma(g,\xi))\in[0,T]}} f(g\xi,g\eta,gw)=
\int_{\dbG\times\Omega} f \dd \BMSm\times\omega, \]
which holds for every function $f$, by linearity.
\end{proof}

\begin{proof}[Proof of Corollary~\ref{C:avaragescheme}]
We note that by Proposition~\ref{P:invariantBMS}, 
\[ \int_\dbG e^{-F(\xi,\eta)} \dd\BMSm= \int_\dbG 1\dd \nu\times\nu =1 \]
and 
\[
		F(\xi,\eta)=\delta_\Gamma\cdot\Gprod{\xi}{\eta}{o}+O(1),
	\]
thus $\psi(\xi,\eta)=e^{-\delta_\Gamma\cdot\Gprod{\xi}{\eta}{o}}$
is in $L^1(\dbG,\BMSm)$ and it has a positive integral.
We apply Theorem~\ref{T:generg} twice,
once for the function $f=\psi\cdot\phi$ and once for the function $f=\psi\cdot 1$,
thus obtain a $\BMSm\times\omega$-full measure subset of $\dbG\times\Omega$ 
for every point of which we have both 
\[ \lim_{T\to\infty}\frac{1}{T}\cdot \sum_{\setdef{g\in\Gamma}{\sigma(g,\eta)\in[0,T]}} \psi(g\xi,g\eta)\cdot\phi(gw)=
\int_{\dbG} \psi \dd \BMSm \cdot
\int_{\Omega} \phi \dd \omega \]
and
\[ \lim_{T\to\infty}\frac{1}{T}\cdot \sum_{\setdef{g\in\Gamma}{\sigma(g,\eta)\in[0,T]}} \psi(g\xi,g\eta)\cdot 1=
\int_{\dbG} \psi \dd \BMSm. \]
The corollary follows clearly by applying Fubini's theorem.
\end{proof}


\section{Further properties of the measured geodesic flow} \label{sec:mackey}

In this section we discuss further properties of the measured geodesic flow $(X,\BMSm,\phi^\bbR)$
and prove Corollaries~\ref{cor:Mackey}, \ref{cor:relmetric} and \ref{C:induced}.
We will also introduce some further constructions extending the scope of Diagram~(\ref{e:quotients}), see Diagram~(\ref{e:finaldiagram}) below.

Recall how we presented the $\Gamma$-action on the space $\dbG\times \bbR$ using the 
cocycle $\tau$ introduced in Equation~(\ref{e:rho-and-F}), namely 
\[ 
    \begin{split}
        \tau(g,\xi,\eta) &=\frac{\rho(g,\eta)-\rho(g,\xi)}{2}\\
        &=\rho(g,\eta)-\frac{1}{2}\nabla_g F(\xi,\eta)\\
        &=-\rho(g,\xi)+\frac{1}{2}\nabla_g F(\xi,\eta),
    \end{split} 
 \]
where $\rho$ is the Radon-Nikodym cocycle associated with the $\Gamma$-action on $(\bG,\PSnu)$
and $F$ is the function introduced in Equation~(\ref{eq:F}).
We have chosen this particular cocycle representative in order to emphasize the flip action on $\dbG\times \bbR$ and on its quotient space $X$.
However, Equation~(\ref{e:rho-and-F}) shows that in fact $\tau$ is cohomologous to the cocycles
\[ 
    -\rho_1,\rho_2:\Gamma\times \dbG \to \bbR, \qquad 
    -\rho_i(g,\xi_1,\xi_2)=-\rho(g,\xi_i), 
\]
as $\nabla_g F(\xi,\eta)$ is, by definition, a coboundary.

\medskip

In this section we will follow another convention, considering the $\Gamma$-action on $\dbG\times \bbR$ using the cocycle $\rho_2$ defined above,
see Remark~\ref{R:other-cocycles}.
As $\rho_2$ is cohomologous to $\tau$ this is just a different presentation of the same object.
Note that $\Gamma$ acts also on the space $\bG\times \bbR$ by
\[ 
    g:(\eta,t)\ \mapsto\ (g\eta,t+\rho(g,\eta)). 
\]
and the projection map $\pr_2:\dbG \to \bG$, $(\xi,\eta)\mapsto \eta$ induces a map $\dbG\times \bbR \to \bG\times \bbR$,
which is $\bbR\times \Gamma$-equiavariant under our new convention.

We have constructed the space $X$ as the space of $\Gamma$-orbits in $\dbG\times \bbR$.
We can interpret this construction as the Mackey range of the cocycle $\rho_2$,
namely the space of the ergodic components of the $\Gamma$-action on  $(\dbG\times\bbR,\BMSm\times\Leb)$,
taking into account that the space of orbits is indeed the space of ergodic components, by Proposition~\ref{P:mgf}.
Next we define similarly the space $Z$ to be the Mackey range of the cocycle $\rho$,
that is the space of ergodic components of the $\Gamma$-action on $(\bG\times\bbR,\PSnu\times\Leb)$,
endowed with the quotient $\bbR$-action, which we denote $\psi^\bbR$, and the quotient measure class.
We emphasize that typically $Z$ is not the space of orbits for the $\Gamma$-action on $\bG\times\bbR$.
In fact, often times this action is ergodic, in which case $Z$ is a singleton.
By construction, we get a natural map 
\[
    \overline{\pr}_2:X\overto{} Z
\]
making Diagram~(\ref{e:finaldiagram}) below commutative.
We endow $Z$ with the measure $\Mm:={\overline{\pr}_2}_*(\BMm)$, that is the push forward of the measure $\BMm$ under this map
(we use ``M" here to honor Mackey).
We note that this is an ergodic $\psi^\bbR$-invariant probability measure
and we call the probability measure preserving system $(Z,\Mm,\psi^\bbR)$ 
{\em the Mackey space of $(\Gamma,[d])$}.

\begin{equation}\label{e:finaldiagram}
	\begin{tikzcd}
		 &  (\dbG\times\bbR,\BMSm\times\Leb) \arrow[dd,"\pr_2\times \id_\bbR"] \arrow[dl,"p"'] \arrow[dr,"q"] & \\
 (X,\BMm) \arrow[dd,"\overline{\pr}_2"] & & (\dbG,\BMSm) \arrow[dd,"\pr_2"] \\
& (\bG\times\bbR,\PSnu\times\Leb) \arrow[dl,"\bar{p}"'] \arrow[dr,"\bar{q}"] & \\
(Z,\Mm) & & (\bG,\PSnu)
	\end{tikzcd}
\end{equation}

In the notation of \S\ref{subsec:mackeyrange}, $\rho_2=\pr_2^*\rho$ and Diagram~(\ref{e:finaldiagram}) is an instance of 
Diagram~(\ref{e:mackeyrange}).

\begin{proof}[Proof of Corollary~\ref{cor:Mackey}]
By Proposition~\ref{prop:mackeyrange} ,$\overline{\pr}_2(X,\BMm) \to (Z,\Mm)$ 
is relatively $\bbR$-metrically ergodic. 
It is thus relatively weakly mixing by Equation (\ref{eq:rel}).
\end{proof}

\begin{remark} \label{rem:important}
(a) Diagram~(\ref{e:finaldiagram}) summarizes the relations between the main objects that 
we have considered throughout this paper.
One should note that the spaces on the right hand side are endowed with a 
$\Gamma$-action, the spaces on the left hand side are endowed
with an $\bbR$-action and the spaces on the middle are endowed with a 
commuting action of both groups.
Moreover, the three spaces on the top are endowed with a flip action, 
that is a $\mathbb{Z}/2$-action, which commutes with the $\Gamma$-actions 
and normalizes the $\bbR$-actions, reversing the time direction.
Thus we may precompose each vertical map by the flip action on the top. 
In particular, we get a mysterious time reversing map $X\to Z$.

(b) The vertical maps are all metrically ergodic with respect to the respective actions.
Conjecturally, the flow $Z$ is $\bbR$-homogenous (that is, a rotation or trivial) and $\overline{\pr}_2$ is relatively mixing (and not merely weakly mixing,
as guaranteed by Corollary~\ref{cor:Mackey}).
The maps $q$ and $\bar{q}$ are quotient maps on the spaces of $\bbR$-orbits, thus they come equipped with a corresponding natural cohomology classes, represented correspondingly by $\rho$ (or $\tau$) and $\rho_2$ and $\gamma$. 
The map $p$ is a quotient map on the spaces of $\Gamma$-orbits, 
thus it comes equipped with a corresponding natural cohomology classes, represented by $\gamma$. 
\end{remark}
 
\begin{proof}[Proof of Corollary~\ref{cor:relmetric}]
We will use the cocycle $\gamma$ defined in Lemma~\ref{L:like-cocycle} and Remark~\ref{R:ess-free}.
The result will follow for every cohomologous cocycle.
The isomorphism
\[ X\times \Gamma \simeq \hat{X}\times \Gamma \simeq \dbG\times \bbR, \quad (x,g) \mapsto (\hat{x},g) \mapsto g\hat{x}, \]
gives a natural projection $\dbG\times \bbR \to \Gamma$, $(\xi,\eta,t)\mapsto g_{\xi,\eta,t}$.
 
Given an isometric action of $\Gamma$ on a separable metric space $S$ and a $\gamma$-equivariant map $f:X\to S$,
that is a map satisfying for every $t\in \bbR$ and a.e $x\in X$, 
\[ f(\phi^t x)=\gamma_{t,x}f(x), \]
we define $\tilde{f}:\dbG\times \bbR \to S$ by $\tilde{f}(\xi,\eta,t)=g_{\xi,\eta,t}f(p(\xi,\eta,t)))$.
One easily checks that $\tilde{f}$ is $\Gamma$-equivariant and $\bbR$-invariant,
thus it gives a well defined $\Gamma$-equivariant map $\bar{f}:\dbG \to S$.
By Theorem~\ref{T:rSAT} this map is essentially constant and its essential image is a $\Gamma$-fixed point in $S$.
We conclude that $f$ is essentially constant and its essential image is a $\Gamma$-fixed point in $S$.
\end{proof}

\begin{proof}[Proof of Corollary~\ref{C:induced}]\hfill{}\\
	Let $(\Omega,\omega,\Gamma)$ be an ergodic p.m.p. system and $(X\times \Omega,\BMm\times\omega,\phi^\bbR_\gamma)$ the induced flow.
	To show ergodicity of the latter flow, consider a measurable integrable $\phi_\gamma^\bbR$-invariant function 
	\[
		F:X\times\Omega\to\bbR.
	\]
By Fubini it gives a map $\gamma$-equivariant $f:X\to L^1(\Omega,\omega)$.
Specializing to $S=L^1(\Omega,\omega)$ in Corollary~\ref{cor:relmetric} we get that $f$ is essentially constant and its image is a fixed point.
By the ergodicty of $\Omega$, this fixed point is a constant function. We conclude that $F$ is essentially constant.
\end{proof}


\end{document}